\documentclass[12pt,a4paper]{amsart}
\usepackage{amssymb,enumerate,amsmath}

\usepackage[pagebackref]{hyperref}

\newcommand{\FF}{{\mathbb{F}}}
\newcommand{\QQ}{{\mathbb{Q}}}

\newcommand{\bB}{\mathbf B}
\newcommand{\bF}{\mathbf F}
\newcommand{\bG}{\mathbf G}
\newcommand{\bH}{\mathbf H}
\newcommand{\bK}{\mathbf K}
\newcommand{\bL}{\mathbf L}
\newcommand{\bN}{\mathbf N}
\newcommand{\bO}{{\mathbf O}}
\newcommand{\bT}{\mathbf T}
\newcommand{\bU}{\mathbf U}
\newcommand{\bZ}{\mathbf Z}

\newcommand{\cE} {\mathcal{E}}

\newcommand{\fA} {\mathfrak A}
\newcommand{\fS} {\mathfrak S}

\newcommand{\Aut}{{\operatorname{Aut}}}
\newcommand{\diag}{{\operatorname{diag}}}
\newcommand{\Ind}{{\operatorname{Ind}}}
\newcommand{\Inn}{{\operatorname{Inn}}}
\newcommand{\Irr}{{\operatorname{Irr}}}
\newcommand{\mh}{{\operatorname{mh}}}
\newcommand{\Out}{{\operatorname{Out}}}
\newcommand{\SC}{{\operatorname{sc}}}
\newcommand{\tA}{\operatorname{A}}
\newcommand{\tB}{\operatorname{B}}
\newcommand{\tC}{\operatorname{C}}
\newcommand{\tD}{\operatorname{D}}
\newcommand{\tE}{\operatorname{E}}
\newcommand{\tF}{\operatorname{F}}
\newcommand{\tG}{\operatorname{G}}

\newcommand{\GL}{\operatorname{GL}}
\newcommand{\SL}{\operatorname{SL}}
\newcommand{\PGL}{\operatorname{PGL}}
\newcommand{\PSL}{\operatorname{PSL}}
\newcommand{\SU}{\operatorname{SU}}
\newcommand{\PSU}{\operatorname{PSU}}

\newcommand{\wt}{\widetilde}

\newcommand{\tw}[1]{{}^#1\!}
\newcommand{{\tchi}}{\tilde\chi}
\newcommand{{\tbG}}{{\widetilde{\bG}}}
\newcommand{{\tiG}}{{\widetilde{G}}}
\renewcommand{\pmod}[1]{~({\rm mod}~#1)}

\let\al=\alpha
\let\eps=\epsilon
\let\ga=\gamma
\let\la=\lambda
\let\si=\sigma
\let\vhi=\varphi
\let\ze=\zeta

\newtheorem{thm}{Theorem}[section]
\newtheorem{lem}[thm]{Lemma}
\newtheorem{cor}[thm]{Corollary}
\newtheorem{prop}[thm]{Proposition}

\newtheorem{thmA}{Theorem}
\newtheorem{conjA}[thmA]{Conjecture}

\theoremstyle{definition}
\newtheorem{rem}[thm]{Remark}

\numberwithin{equation}{section}

\usepackage{hyperref}

\begin{document}

\title[Minimal heights in blocks]{On minimal positive heights for blocks\\ of almost quasi-simple groups}

\author{Gunter Malle}
\address[G. Malle]{FB Mathematik, RPTU, Postfach 3049,
  67653 Kaisers\-lautern, Germany.}
\email{malle@mathematik.uni-kl.de}

\author{A. A. Schaeffer Fry}
\address[A. A. Schaeffer Fry]{Dept. Mathematics - University of Denver, Denver
 CO 80210, USA}
\email{mandi.schaefferfry@du.edu}

\begin{abstract}
The Eaton--Moret\'o conjecture extends the recently-proven Brauer height zero
conjecture to blocks with non-abelian defect group, positing equality between
the minimal positive heights of a block of a finite group and its defect
group.  Here we provide further evidence for the inequality in this conjecture
that is not implied by Dade's conjecture.
Specifically, we consider minimal counter-examples and show that these cannot
be found among almost quasi-simple groups for $p\ge5$. Along the way, we
observe that most such blocks have minimal positive height equal to~1. 
\end{abstract}

\thanks{The first author gratefully acknowledges financial support by the DFG
-- Project-ID 286237555.
The second author gratefully acknowledges support from the National Science
Foundation, Award No. DMS-2100912, and her former institution, Metropolitan
State University of Denver, which held the award and allowed her to serve as PI}



\date{\today}

\maketitle


\section{Introduction}   \label{sec:intro}
In this paper, we are concerned with a conjecture of Charles Eaton and Alex
Moret\'o \cite{EM14} on the $p$-parts of character degrees of finite groups,
where $p$ is a prime number. This conjecture can be considered as an extension
of Richard Brauer's height zero conjecture (whose proof was recently completed
in \cite{MNST}) to the case of non-abelian defect groups.

In order to formulate the Eaton--Moret{\'o} conjecture, we introduce some
notation. For $D$ a (finite) $p$-group we write
$$\mh(D):=\min\{\infty,\log_p(\chi(1))\mid \chi\in\Irr(D),\,\chi(1)>1\}$$
for the minimal height of a non-linear complex irreducible character of $D$,
if such exists, and $\infty$ otherwise. Note that the latter occurs exactly
when  $D$ is abelian. Analogously, for $B$ a Brauer $p$-block of a finite
group~$G$, we let $\mh(B)$ denote the minimal positive height of an irreducible
complex character in $B$, if such a character exists, and $\infty$ otherwise.
Here recall that the height $h$ of $\chi\in\Irr(B)$ is defined by
$(\chi(1)|D|/|G|)_p=p^h$, where $D$ is any defect group of $B$.

Brauer's height zero conjecture (now a theorem) states that for a
$p$-block $B$ of a finite group with defect group $D$, all characters in $B$
have height zero if and only if $D$ is abelian. In the terminology above, this
can be stated as $\mh(B)=\infty$ if and only if $\mh(D)=\infty$. In 2014,
Eaton and Moret\'o \cite{EM14} conjectured that these invariants agree in
general:

\begin{conjA}   \label{conj}
 Let $B$ be a $p$-block of a finite group with defect group $D$.
 Then $\mh(B)=\mh(D)$.

\end{conjA}

This has meanwhile been verified in several cases and for various classes of
groups and types of blocks (see e.g.~\cite{EM14,BM15,FLZ,MMR,N24}) but remains
open in general, even for solvable groups.  In \cite{EM14} it was further shown
that Dade's projective conjecture, widely believed to be true and which has
been reduced to a question on simple groups, would imply one part of
Conjecture~\ref{conj}, namely that $\mh(B)\ge\mh(D)$. For this reason,
providing further evidence for the reverse inequality seems the most pertinent.
The work in \cite{MMR, N24} indicates how difficult an eventual reduction to
simple groups might be for this direction of the conjecture, even for the
principal block case.

In this paper, we focus on a minimal counter-example to the
``$\mh(B)\leq\mh(D)$''-direction of Conjecture \ref{conj}. We obtain the
following results:

\begin{thmA}   \label{thm:quasi}
 Let $B$ be a $p$-block of a finite quasi-simple group. Then $B$ is not a
 minimal counter-example to the direction $\mh(B)\le\mh(D)$ of
 Conjecture~\ref{conj}.
\end{thmA}

We remark that this result for the case $p>5$ was stated in
\cite[Cor.~1.5]{FLZ} (except for proper covering groups of sporadic simple
groups). We also consider extensions by $p$-automorphisms (which will most
likely be relevant even if some kind of reduction of Conjecture~\ref{conj} to
decorated simple groups is available):

\begin{thmA}   \label{thm:aut}
 Let $p\ge5$ and let $B$ be a $p$-block of a finite group $A$ with $\bF^*(A)$
 quasi-simple. Then $B$ is not a minimal counter-example to the direction
 $\mh(B)\le\mh(D)$ of Conjecture~\ref{conj}.
\end{thmA}

In fact, we also prove partial results for $p=2,3$ on this question. Further,
in the case that $\bF^*(A)$ is quasi-simple not of Lie type in non-defining
characteristic, we obtain the full Conjecture~\ref{conj} for $A$ for all $p$.

Along the way, we find that in the situation of Theorem \ref{thm:aut}, it is
often the case that $\mh(B)=1$ when $B$ has non-abelian defect groups. This is
true for alternating groups, sporadic groups (with just three exceptions for
covering groups when $p=3$), and the Tits group $\tw{2}\tF_4(2)'$, but we also
see that it holds in the so-called quasi-isolated blocks for groups of Lie type.

\begin{thmA}   \label{thm:qiIntro}
 Let $p\ge5$ and let $A$ be a finite group such that $S=\bF^*(A)$ is
 quasi-simple of Lie type in characteristic distinct from $p$. Then any
 $p$-block $B$ of $A$ of non-abelian defect covering a quasi-isolated block of
 $S$ satisfies $\mh(B)=1$.
\end{thmA}

Again, we also prove partial results for $p=2,3$ in the situation of
Theorem~\ref{thm:qiIntro}. Our investigations lead us to believe that
the conclusion of Theorem~\ref{thm:qiIntro} should actually hold for all blocks
of the groups $A$ as above, a question we hope to come back to in the future.
\medskip

The paper is structured as follows. In Section \ref{sec:elem}, we provide some
initial reductions and observations that will be helpful throughout. Groups not
of Lie type and groups of Lie type defined in characteristic~$p$ are
discussed in Sections~\ref{sec:nonlie} and~\ref{sec:defchar}, respectively.
Finally, in Section~\ref{sec:crosschar}, we treat the groups of Lie type
defined in characteristic distinct from $p$ and complete the proofs of
Theorems~\ref{thm:quasi}, \ref{thm:aut}, and~\ref{thm:qiIntro}.

\section{Elementary observations and reductions}   \label{sec:elem}

Given a $p$-block $B$ of a finite group, we let $d(B)$ denote the defect of the
block. That is, $d(B)$ is the integer such that $|D|=p^{d(B)}$, where $D$ is
any defect group of $B$. The following will be used frequently.

\begin{lem}   \label{lem:pfactor}
 Let $N\unlhd G$ be finite groups with $G/N$ a $p$-group of order $p^a$ and let
 $b$ be a $p$-block of $N$.  Then there is a unique $p$-block $B$ of $G$
 covering $N$. If $b$ is further $G$-invariant, then $d(B)=d(b)+a$.
\end{lem}

\begin{proof}
The first statement is \cite[Cor.~9.6]{N98}, and the second follows immediately
from \cite[Thm~9.17]{N98}.
\end{proof}

\begin{lem}   \label{lem:not-invariant}
 Let $N\unlhd G$ be finite groups, and $B$ a $p$-block of $N$ with stabiliser
 $G_B=N$ in $G$. Then Conjecture~\ref{conj} holds for the unique block $\hat B$
 of $G$ covering $B$ if and only if it holds for $B$.
\end{lem}

\begin{proof}
This is an immediate consequence of \cite[Thm~9.14]{N98}.
\end{proof}

The next observation shows we may restrict attention to groups $G$ with
$G=O^{p'}(G)$:

\begin{lem}   \label{lem:Op'}
 Let $N\unlhd G$ be finite groups with $[G:N]$ prime to $p$. Then
 Conjecture~\ref{conj} holds for the $p$-blocks of $G$ if and only if it holds
 for the $p$-blocks of $N$. Here, if the block $\hat B$ of $G$ covers the block
 $B$ of $N$, then $\mh(B)=\mh(\hat B)$.
\end{lem}

\begin{proof}
Let $B$ be a $p$-block of $N$ with defect group $D$. Then any block $\hat B$ of
$G$ covering $B$ has the same defect groups. Moreover, if $\chi\in\Irr(B)$ has
height $\mh(B)$ then all constituents of the induced character $\chi^G$ have the
same height as $\chi$ by Clifford theory.  Thus $\hat B$ satisfies
Conjecture~\ref{conj} if $B$ does. Conversely, for $\hat B$ a block of $G$ let
$B$ be a
block of $N$ covered by $\hat B$. Then again $B,\hat B$ have the same defect
groups, and the restriction of any $\chi\in\Irr(\hat B)$ to $N$ has a
constituent of the same height. Since this constituent lies in a conjugate
of~$B$, we see that $B$ also contains a constituent of that height.
\end{proof}

\begin{lem} \label{lem:centralprod}
 Let $G=NZ$ be the product of commuting subgroups such that $Z$ is abelian and
 $N\cap Z$ is a central $p'$-subgroup. Let $\hat B$ be a $p$-block of $G$
 covering a block $B$ of~$N$. Then $\mh(B)=\mh(\hat B)$. Further,
 Conjecture~\ref{conj} holds for $B$ if and only if it holds for~$\hat B$.
\end{lem}

\begin{proof}
Since $G$ can be viewed as a quotient of $N\times Z$ by a central
$p'$-subgroup, we may identify $\Irr(\hat B)$ with $\Irr(B\otimes B')$ for some
block $B'$ of $Z$ using \cite[Thm~9.9]{N98}. Further, any defect group
$\hat D$ of $\hat B$ is isomorphic to one of $B\otimes B'$, and is of the form
$D\times D'$ with $D$ a defect group for $B$ and $D'$ a defect group for $B'$. 
Since all characters in $\Irr(B')$ and in $\Irr(D')$ are linear, we conclude
$\mh(B)=\mh(\hat B)$ and $\mh(D)=\mh(\hat D)$.
\end{proof}

\subsection{Metacyclic defect groups}

We next make some useful observations about certain blocks with metacyclic
defect groups.

\begin{lem}   \label{lem:meta}
 Let $P$ be a non-abelian metacyclic $p$-group, with $p>2$. Then $\mh(P)=1$.
\end{lem}

\begin{proof}
Let $N\unlhd P$ be a cyclic normal subgroup with $P/N$ cyclic. Then any
irreducible character of $N$ extends to its stabiliser in $P$. As $P$ is
non-abelian, there is a character $\theta\in\Irr(N)$ that is not $P$-invariant.
Let $p^a>1$ be the length of the orbit of $\theta$ under $P$. Then $\theta^p$
lies in an orbit of length $p^{a-1}$ and thus, some power $\chi$ of $\theta$
has stabiliser of index~$p$ in $P$. Any character of $P$ above $\chi$ then has
height~1.
\end{proof}

\begin{prop}   \label{prop:meta}
 Let $G$ be a finite group and let $p$ be a prime dividing $|G|$. Let
 $\si\in\Aut(G)\setminus\Inn(G)$ be an automorphism of $p$-power order and
 define $A:=G\langle\si\rangle$. If $B_A$ is a $p$-block of $A$ covering a
 block $B$ of $G$ with cyclic defect groups, then $\mh(B_A)=\mh(D_A)$, where
 $D_A$ is a defect group for $B_A$.
\end{prop}

\begin{proof}
Let $D_A$ be a defect group for $B_A$ and let $D:=D_A\cap G$, a (cyclic)
defect group for~$B$. Note that $D_A/D\cong D_AG/G\leq A/G$ is cyclic, so that
$D_A$ is metacyclic. If $p=2$, then \cite[Cor.~8.2]{Sam} yields that
$\mh(B_A)=\mh(D_A)$, so we assume $p$ is odd. Note that Lemma \ref{lem:meta}
gives $\mh(D_A)=1$. So, we must show that $\mh(B_A)=1$.

Applying \cite[Thm~9.14(d)]{N98}, we may assume $B$ is stable under $\si$.
Since $A/G$ is cyclic, it suffices to show that there exists a
(height-zero) character of $B$ lying in an orbit of length $p$ under $\si$.

Following \cite[Thm~1 Part~1]{Dad66} and the discussion before, since $B$ is
cyclic, there is a block $b$ of $DC_G(D)=C_G(D)$, unique up to
$N_G(D)$-conjugation, such that $b^G=B'$, where $B'$ is the Brauer
correspondent of $B$ in $N_G(D)$. Then writing $E:=N_G(D)_b$, we have
$[E:C_G(D)]$ is prime to $p$. Further, the block $B$ contains
$(|D|-1)/[E:C_G(D)]$ so-called exceptional characters $\chi_\la\in\Irr(B)$,
labelled by $1\neq \la\in\Irr(D)$, up to $E$-conjugacy. As in the
second-to-last sentence of the proof of \cite[Prop.~5.1]{KoS}, we have
$\chi_\la^\tau=\chi_{\la^\tau}$ for any $\tau\in\langle\si\rangle$. Then
$\chi_\la^\tau=\chi_\la$ if and only if $\la^\tau=\la^g$ for some
$g\in E$. Note that since $D$ is cyclic, $\tau$ and $g$ act on $\Irr(D)$
by powers, and hence these actions commute. But then $\la$ has the same orbit
size under the $p$-elements $\tau$ and $g$, contradicting that
$E/C_G(D)$ has $p'$-order, unless $g\in C_G(D)$.  Hence, we see for
$1\neq \tau\in \langle\si\rangle$, we have $\chi_\la^\tau=\chi_\la$ if and
only if $\la^\tau=\la$. Then choosing $\la$ to have orbit of length $p$ under
$\si$ as in Lemma~\ref{lem:meta}, we obtain $\chi_\la$ in an orbit of length
$p$ under~$\si$.
\end{proof}

\section{Groups not of Lie type}   \label{sec:nonlie}

In this section, we prove our main results for decorated simple groups that are
not of Lie type. The covering groups of alternating and symmetric groups were
shown to satisfy Conjecture~\ref{conj} in \cite[Thm~2.1]{BM15}. Since $\fS_n$
is the full automorphism group of $\fA_n$ when $n\ne6$, the following result
completes the investigation of decorated alternating groups.

\begin{prop}   \label{prop:alt}
 Conjecture~\ref{conj} holds for all groups $A$ with $\bF^*(A)$ a covering
 group of~$\fA_6$. Further, we have $\mh(B)=1$ for the blocks of these groups
 with non-abelian defect groups.
\end{prop}

\begin{proof}
Using Lemma~\ref{lem:Op'}, we may assume $A=O^{p'}(A)$. In view of
\cite[Thm~2.1]{BM15} we only need to discuss faithful 2-blocks of the groups
$A=\fA_6.2_2=\PGL_2(9)$, $\fA_6.2_3=M_{10}$, $2.\fA_6.2_2$, $3.\fA_6.2_2$,
$3.\fA_6.2_3$, and $6.\fA_6.2_2$ in Atlas notation. By inspection of their
character tables in \cite{gap} one sees that all blocks $B$ with non-abelian
defect groups $D$ have $\mh(B)=1=\mh(D)$, as wanted.
\end{proof}

For the sporadic simple groups, Conjecture~\ref{conj} was verified for all
blocks in \cite[Thm~D]{EM14}. Here we consider the decorated groups:

\begin{prop}   \label{prop:spor}
 Conjecture~\ref{conj} holds for all groups $A$ with $\bF^*(A)$ a covering
 group of a sporadic simple group or $\tw2F_4(2)'$.
\end{prop}

\begin{proof}
Using Lemma~\ref{lem:Op'}, we may again assume $A=O^{p'}(A)$. Using \cite{gap}
it turns out that all $p$-blocks $B$ of non-simple groups~$A$ as in the
statement with defect group $D$ of order at least~$p^3$ either have all
characters of height zero (and thus $D$ abelian by Brauer's height zero
conjecture), or have $\mh(B)=1$, or $B$ is the principal 3-block of $A=3.ON$
and $\mh(B)=2$. Except for the last case, it hence suffices to see that
$\mh(D)=1$ (e.g., by showing $\mh(\bar D)=1$ for some factor group $\bar D$
of~$D$).
\par
If $|D|\le p^4$ then clearly $\mh(D)=1$.
Further, if $B$ dominates a block with non-abelian defect of a quotient $\bar A$
of $A$ by a central $p$-subgroup, and the conjecture is verified for the
$p$-blocks of $\bar A$, it also holds for $B$. If $B$ covers a block $B'$ with
the same defect group of a normal subgroup of $A$, again the assertion for $B$
follows from the one for $B'$. 
Next, if $|\bZ(\bF^*(A))|$ is prime to $p$ and $B$ has a Sylow $p$-subgroup
of~$A$ as defect group, the claim follows from the corresponding one for
$A/\bZ(\bF^*(A))$. Since Conjecture~\ref{conj} has been shown for simple
sporadic groups in \cite[Thm~D]{EM14}, we are then left with the following
cases: the principal 2-blocks of $S.2$ for
$$S\in\{M_{12},M_{22},J_2,HS,J_3,McL,He,Suz,ON,Fi_{22},HN,Fi_{24}'\},$$
a 2-block of $HN.2$ of defect~5, and the principal 3-block of $3.ON$.

For $S\in\{M_{12},M_{22},J_2,HS,J_3,McL,He,Suz\}$ the Sylow 2-subgroups of $S.2$
are readily constructed in GAP \cite{gap}, and the claim follows; a Sylow
2-subgroup of $ON.2$ is contained
in a maximal subgroup $4^3.(\PSL_3(2)\times2)$, and Sylow 2-subgroups $\bar D$
of $\PSL_3(2)$ have \hbox{$\mh(\bar D)=1$}; a Sylow 2-subgroup of $Fi_{22}.2$ is
contained in a maximal subgroup $2^{10}.M_{22}.2$, hence has characters of
height~1 by the case of $M_{22}.2$; a Sylow 2-subgroup of $HN.2$ is contained
in a maximal subgroup $2^{3+2+6}.(3\times\PSL_3(2)).2$, and a Sylow 2-subgroup
$\bar D$ of $\PSL_3(2).2$ has order~16, so $\mh(\bar D)=1$; a Sylow
2-subgroup of $Fi_{24}'.2$ is contained in a maximal subgroup $2^{12}.M_{24}$,
and a Sylow 2-subgroup $\bar D$ of $M_{24}$ has $\mh(\bar D)=1$.
The non-principal 2-block of $ON.2$ of defect~5 covers a 2-block of $ON$ with
abelian Sylow 2-subgroups and thus its (non-abelian) defect groups possess
characters of height~1. Using the smallest faithful permutation representation
of $A=3.ON$, one computes using \cite{gap} that $\mh(D)=2$ for a Sylow
3-subgroup $D$ of $A$. This completes the proof of the proposition.
\end{proof}

The principal 3-blocks of $3.ON$, $3.ON.2$ and of $Co_3$ (treated in
\cite{EM14}) are the only blocks~$B$ in the realm of decorated sporadic groups
with $\mh(B)>1$ and non-abelian defect groups.

\begin{prop}   \label{prop:exc cov}
 Let $A$ be such that $\bF^*(A)$ is an exceptional covering group with cyclic
 centre of a simple group of Lie type and $A/\bF^*(A)$ is generated by field
 automorphisms. Then Conjecture~\ref{conj} holds for all blocks of $A$ for all
 primes.
\end{prop}

\begin{proof}
The exceptional Schur multipliers of groups of Lie type are listed in
\cite[Tab.~24.3]{MT}. Note that the only cases for which non-trivial field
automorphisms exist are $\PSL_3(4)$, $\PSU_4(3)$, $\PSU_6(2)$, $\tw2B_2(8)$,
$G_2(4)$ and $\tw2E_6(2)$. Using \cite{gap} one sees that the only cases for
which there is a block with $\mh(B)\ge2$ are for $\PSL_3(4)$, $3.\PSL_3(4)$ and
$3.\PSL_3(4).2_2$ with $p=2$. In all cases coming up, the same kind of
reductions as in the proof of Proposition~\ref{prop:spor} show that the claim
follows once we know it for principal blocks of non-exceptional covering
groups of $\bF^*(A)/\bZ(\bF^*(A))$. These cases were settled in
\cite[Thm~3.9]{BM15} for the defining prime, respectively in
\cite[Prop.~2.2]{FLZ} for non-defining primes.
\end{proof}

\section{Groups of Lie type in defining characteristic}   \label{sec:defchar}

Throughout, let $\bG$ be a simple linear algebraic group over an algebraically
closed field~$k$ of characteristic~$p>0$. Let $\bT\le \bG$ be a maximal torus
contained in a Borel subgroup $\bB\le\bG$ and let $\bU$ be the unipotent
radical of~$\bB$.  We denote by $\Phi$ the root system of~$\bG$ with respect to
$\bT$, and by $\Phi^+\subset\Phi$ the set of positive roots with respect to
$\bB$. Write $\Delta\subset\Phi^+$ for the set of simple roots. For any
$\al\in\Phi^+$, we denote by $\bU_\al$ the corresponding root subgroup in $\bU$
normalised by $\bT$, and we choose an isomorphism
$x_\al:k\rightarrow \bU_\al$ (see e.g.~\cite[Sect.~8]{MT}).   \par
For any graph automorphism $\gamma$ of~$\bG$ and for any power
$q=p^f$, there is a Frobenius endomorphism $F:\bG\to\bG$ with the
property that $F(x_\al(c))=x_{\ga(\al)}(c^q)$ for all $\al\in\Phi$,
$c\in k$ (see \cite[Thm~1.15.2]{GLS}). We define $F_p:\bG\to\bG$,
$x_\al(c)\mapsto x_\al(c^p)$, a Frobenius map with respect to an
$\FF_p$-structure. Note that $F,F_p$ commute in their action on~$\bG$, and
$\bT,\bB$ and $\bU$ are $F$- and $F_p$-stable.
Let further $\bG^*$ be dual to $\bG$ with corresponding Frobenius endomorphisms
again denoted $F,F_p$ and let $G:=\bG^F$, $G^*:=\bG^{*F}$.

It was shown in \cite[Prop.~3.8]{MMR} that Conjecture~\ref{conj} holds for the
principal $p$-blocks of groups $A$ obtained by extending a simple group of Lie
type $S$ in characteristic~$p$ by a field automorphism of $p$-power order. We
now generalise this to all $p$-blocks of such groups.

\begin{prop}   \label{prop:field}
 Let $S$ be a quasi-simple group of Lie type in characteristic~$p$ with
 $\bZ(S)$ of $p'$-order, $\si$ a $p$-power order field automorphism of $S$ and
 $A:=S\langle\si\rangle$. Then Conjecture~\ref{conj} holds for all $p$-blocks
 of $A$.
\end{prop}

\begin{proof}
If $A=S$, or if $B$ is the principal block of $A$ the claim was shown in
\cite[Thm~3.9]{BM15}, respectively in \cite[Prop.~3.8]{MMR}. So now assume $A>S$
and $B$ is a non-principal $p$-block of~$A$. We may also assume $B$ is not of
defect zero, so by \cite{Hum} we must have $\bZ(S)\ne1$. Then there exists a
simple algebraic group $\bG$ of simply connected type
with a Frobenius map $F$ and with an epimorphism $G=\bG^F\to S$ with central
kernel of $p'$-order, and all field automorphisms of $S$ lift to $G$. In
particular, any $p$-block of $S$ is also a $p$-block of~$G$, with isomorphic
defect groups. Thus, without loss, we may assume $S=G$. Let $B_S$ be a $p$-block
of $S$ covered by $B$. By \cite{Hum}, $B_S$ has as defect groups the Sylow
$p$-subgroups of $S$ and $\Irr(B_S)$ consists of all irreducible characters of
$S$ lying above a fixed character $\theta\ne1$ of $\bZ(S)$. By
Lemma~\ref{lem:not-invariant}, we may assume $B_S$, and hence $\theta$, is
$\si$-stable. Since $A/S$ is a $p$-group, then a Sylow $p$-subgroup $P$ of
$A$ is a defect group of~$B$ by Lemma~\ref{lem:pfactor}. For $A>S$, it was
shown in \cite[Prop.~3.8]{MMR} that $\mh(P)=1$. So it remains to prove that
$\mh(B)=1$.

Let $s\in[{G^*}^{\si^p},{G^*}^{\si^p}]$ be a semisimple element as constructed
in the proof of \cite[Prop.~3.8]{MMR} whose $G^*$-conjugacy class lies in an
orbit of length~$p$ under $\si$. Observe that by construction, $|s|$ does not
divide $|\bZ(G)|$, whence $C_{\bG^*}(s)$ is connected. Since $s$ is semisimple,
it lies in a maximal torus of $\bG^{*F}$ and thus its centraliser
$C:=C_{\bG^*}(s)^F$ generates $G^*$ over $[G^*,G^*]$. First assume $\bG$ has
type $\tB_n$ ($n\ge1$), $\tC_n$ ($n\ge3$), $\tD_n$ ($n\ge4$) or $\tE_7$. Since
$\bZ(G)\ne1$ we have $p\ne2$, so $\si$ has odd order and so the group $C^\si$
of fixed points of $C$ under $\si$ contains a Sylow 2-subgroup of~$C$, and
$[G^*:[G^*,G^*]]\in\{2,4\}$. Hence any coset of $[G^*,G^*]$ in $G^*$
contains a $\si$-stable (2-)element $t$ of $C_{G^*}(s)$, so in particular so
does the coset corresponding to $\theta$ under the duality map
\cite[Prop.~2.5.20]{GM20}. Then the class of $st$ also lies in an orbit of
length~$p$ under~$\si$. Let $\chi$ be a semisimple character of $S$ labelled
by $st$ (see e.g.~\cite[Def.~2.6.9]{GM20}). Then $\chi$ has $p$-height~0, lies
in $B_S$ by \cite{Hum}, and by Clifford theory all characters of~$A$ above
$\chi$ have height~1 and lie in~$B$, proving $\mh(B)=1$ in this case.

Now assume $\bG$ has type $\tE_6$, so $[G^*:[G^*,G^*]]=3$ and $p\ne3$. If
$C^\si$ contains a Sylow 3-subgroup of $C$ we can argue exactly as before.
Otherwise, we necessarily have $p=2$, and $\si$ makes an orbit of length two
on elements $t$ of 3-power order generating $C$ over $C\cap[G^*,G^*]$,
implying that $\theta$ is not $\si$-invariant, contrary to our assumption.

Finally, let $\bG$ be of type $\tA_{n-1}$ with $n\ge3$, so $G=\SL_n(\eps q)$
with $\eps\in\{\pm1\}$.
By assumption the coset of $[G^*,G^*]$ in $G^*$ corresponding to $\theta$ is
$\si$-stable. Let $\ze\in\FF_q^\times$ such that the image $s$ in
$\PGL_n(q)=G^*$ of $\diag(\ze,1,\ldots,1)\in\GL_n(q)$ lies in this coset. Then,
$s$ lies in the maximally split torus $T^*$ which we may assume to be
$\si$-stable, and the field automorphism~$\si$ acts by sending $s$ to some
power. Now looking at the eigenvalues we see that $s$ cannot be conjugate in
$G^*$ to any of its distinct powers, so in fact $s$ must be $\si$-invariant.
Write $T':= T^*\cap[G^*,G^*]$, a maximally split torus of $[G^*,G^*]$. Then
clearly $T'^\si$ is a proper subgroup of $T'^{\si^p}$ and we let
$t\in T'^{\si^p}\setminus T'^\si$. Then the class of $st$ lies in an orbit of
length~$p$ under~$\si$. Taking a semisimple character corresponding to this
class in $B_S$ then provides a character of height~1 above it in $B$.
The case of $\SU_n(q)$ is analogous.
\end{proof}

\begin{thm}   \label{thm:defchar}
 Conjecture~\ref{conj} holds for all $p$-blocks of groups $A$ with $\bF^*(A)$ a
 covering group of a simple group of Lie type in characteristic~$p$.
\end{thm}

\begin{proof}
Once again, note that we may assume $A=O^{p'}(A)$, by Lemma~\ref{lem:Op'}.
Let $S:=\bF^*(A)$. If $S$ is an exceptional covering group of $S/\bZ(S)$, the
claim is Proposition~\ref{prop:exc cov}, so we may assume $|\bZ(S)|$ is prime
to~$p$. Furthermore, by the result of \cite[Prop.~3.8]{MMR}, it suffices to
consider faithful blocks of groups $S$ with $\bZ(S)\ne1$.
The outer automorphisms of $p$-power order of simple groups of Lie type $S$ in
characteristic~$p$ are field, graph and graph-field automorphisms (see e.g.
\cite[Thm~24.24]{MT}). In particular, for $p\ge5$ only field automorphisms
occur. Furthermore, if $S$ is not of type $\tA$, these are central in $\Out(S)$
in which case our claim follows from Proposition~\ref{prop:field}.

Next assume $S$ is of type $\tA$ or $\tE_6$ and $A$ is a subgroup of the
group of field and diagonal automorphisms, and keep the notation $B_S$, etc.\
as in Proposition~\ref{prop:field}. As diagonal automorphisms have order prime
to $p$, then $A/S$ is a semidirect product of a group of diagonal automorphisms
by a group of field automorphisms, generated by $\si$, say. Let $S<A'\le A$ be
the extensions of $S$ in $A$ generated by the diagonal automorphisms only. By a
well-known result of Brauer (see e.g.\cite[Probl.~4.8]{N98}), the $p$-blocks of
maximal defect of any group $H$ between $S$ and $A'$ are in bijection with the
set of $p'$-elements centralising a Sylow $p$-subgroup of $H$ and thus with
$\bZ(S)$; in particular, there is a unique block of $H$ above $B_S$. Let
$\chi\in\Irr(B_S)$ be the character constructed in the proof of
Proposition~\ref{prop:field}, so $\chi$ is $\si^p$-invariant but not
$\si$-stable. Now $|\Irr(A'|\chi)|$ is a divisor of $[A':S]$, hence prime
to~$p$. So it contains a $\si^p$-invariant character $\chi'$. By the
construction of $\chi$, $\chi^\si$ lies in a different Lusztig series, so
$\chi'$ cannot be $\si$-invariant and thus any character of $A$ above $\chi'$
is as required.
\par
For $p=3$ the only groups $S$ with additional $p$-automorphisms are those of
untwisted type~$\tD_4$, for which $\Out(S)\cong\fS_4\times C_f$.
Thus $A/S\le \fA_4\times C_{3^k}$ with $3^k|f$. By Proposition~\ref{prop:field}
we may assume that $A$ induces some non-field automorphisms of 3-power order.
Now note that non-trivial graph and graph-field automorphisms act fixed-point
freely on the centre $\bZ(G)\cong C_2^2$ of our group $G$ of simply connected
type. Hence $A$ cannot fix any $p$-block of~$G$ above a non-trivial central
character. Thus we are done by the above reduction.

If $p=2$ then additional graph automorphisms of order~2 exist for types
$\tA_n$, $\tD_n$ and $\tE_6$, as well as exceptional graph automorphisms for
types $\tB_2$ and $\tF_4$. Now the groups of types $\tD_n$, $\tB_2$ and
$\tF_4$ do not have proper generic covering groups in characteristic~$p=2$,
so they are already handled by \cite[Prop.~3.8]{MMR}.
For $\SL_n(\eps q)$ and $\tE_6(\eps q)$ the graph- and graph-field automorphisms
act by inversion on the centre (of odd order) of the simply connected group, so
again there are no invariant non-trivial characters of $\bZ(G)$ in
characteristic~$2$.
\end{proof}

\section{Groups of Lie type in non-defining characteristic}   \label{sec:crosschar}
With the results of the previous sections, note that the only remaining
obstacle to Theorems~\ref{thm:quasi} and~\ref{thm:aut} is the case of groups of
Lie type in non-defining characteristic. Throughout this section, we fix the
following setup. Let $\bG$ be a simple linear algebraic group of
simply connected type with a Steinberg map $F:\bG\to\bG$ and set $G:=\bG^F$ the
corresponding finite reductive group. We let $q$ denote the absolute
value of the eigenvalues of $F$ on the character group of an $F$-stable maximal
torus of $\bG$; so $F^2$ is a Frobenius map with respect to an
$\FF_{q^2}$-structure. We also let $(\bG^*,F)$ be in duality with $(\bG,F)$ and
set $G^*:=\bG^{*F}$. We let $\ell$ be a prime not dividing $q^2$.

\subsection{Blocks and field automorphisms}

We will need the following extension to Steinberg maps of a result shown for
Frobenius maps in \cite{MNST}:

\begin{prop}   \label{prop:4.8}
 Let $\si$ be a field automorphism of $G=\bG^F$ of order $\ell$, and let
 $\ga =\si\tau$, where $\tau$ is an inner-diagonal automorphism of $G$. Then
 any $\ell$-block of $G$ of non-central defect contains characters in a
 Lusztig series that is not $\ga$-stable.
\end{prop}

\begin{proof}
When $F$ is a Frobenius map, this was shown in \cite[Prop.~4.8]{MNST}. We argue
that the same proof goes through for Steinberg maps. Thus we may assume $\bG$
is of type $\tB_2$, $\tG_2$, or $\tF_4$ in characteristic $p=2,3,2$
respectively. In particular, $G$ has no non-inner diagonal automorphisms and so
$\ga=\si$. Moreover, $G$ also has no field automorphisms of even order, so
$\ell$ is necessarily odd.

The $\ell$-blocks of $G$ are associated to $\Phi$-cuspidal pairs $(\bL,\la)$,
where $\Phi$ is a cyclotomic polynomial over $\QQ(\sqrt{p})$ (as in
\cite[3.5.3]{GM20}) such that $\ell$ divides $\Phi(q)$. By assumption, up to
conjugation, $\si$ is induced by a Steinberg map $F_0$ of $\bG$ such that
$F=F_0^\ell$, so $F_0^2$ induces an $\FF_{q_0^2}$-structure, with $q=q_0^\ell$.
Then $\ell$ also divides $\Phi(q_0)$, whence any $\Phi$-torus of
$(\bG, F_0)$ is also a $\Phi$-torus of ($\bG,F)$, and conversely, any
$F_0$-stable $\Phi$-torus of $(\bG,F)$ is also a $\Phi$-torus of $(\bG, F_0)$.

We can now copy the argument given in the proof of \cite[Prop.~4.8]{MNST} to
construct a cyclic $\si^*$-stable $\Phi$-torus $\bT_1$ of $\bG^*$ and an
$\ell$-element
$t\in\bT_1^F$ that is not $\si^*$-stable. Note that the construction of the
automorphism $\si^*$ of $\bG^*$ in \cite{Tay} does not require $F$ to be a
Frobenius map. Using the table in \cite[Tab.~3.3]{GM20} of $\Phi$-split Levi
subgroups (and the classification of isolated elements when $\ell$ is bad for
$\bG$, which only happens for $\ell=3$ and $\bG$ of type $\tF_4$), we see that
$C_{\bG^*}((\bT_1^F)_\ell) =C_{\bG^*}((\bT_1)_\Phi)$. The
rest of the argument is then completed as in \cite{MNST}.
\end{proof}

The following consequence of Proposition~\ref{prop:4.8} will allow us to deal
with blocks of $G$ with abelian defect.

\begin{prop}   \label{prop:abelian}
 Let $\ell\ge3$ and let $A=G\langle\si\rangle$ where $\si$ is a non-trivial
 $\ell$-power order field automorphism of~$G$. Assume that $B$ is a
 $\si$-stable $\ell$-block of $G$ with non-central abelian defect. Then
 $\mh(B_A)=1$ for the $\ell$-block $B_A$ of $A$ covering $B$.  If further
 $G=\SL_n(\eps q)$ and $B$ has non-cyclic defect groups, then a character in
 $B_A$ with height one can be chosen to lie above a character of $G$ trivial
 on $\bZ(G)_\ell$. 
\end{prop}

\begin{proof}
Since $\si\ne1$ by assumption, by Proposition~\ref{prop:4.8} there are
characters
of~$G$ in $B$ not stable under $\si$, so since $A/G$ is an $\ell$-group, the
unique block $B_A$ of $A$ covering $B$ contains characters of positive height.
More specifically, let $s\in G^*$ such that $\Irr(B)\subseteq\cE_\ell(G,s)$.
In the proof of \cite[Prop.~4.8]{MNST}, respectively Proposition~\ref{prop:4.8},
we constructed a $\si^*$-stable $\Phi$-torus $\bT_1^*\le\bG^*$ of rank~1 (where
$\Phi=\Phi_e$ with $e=e_\ell(q)$ when $F$ is a Frobenius map) and showed that
the automorphism $\si^*$ of $G^*$
of $\ell$-power order induced by $\si$ acts non-trivially on the Sylow
$\ell$-subgroup $P$ of $\bT_1^{*F}$. Thus there exists an $\ell$-element
$t\in P$ whose $\si^*$-orbit has length~$\ell$. As argued in the proof
of \cite[Prop.~4.8]{MNST}, we may assume that $t\in C_{G^*}(s)$. Then by
\cite[Lemma~3.13]{MNST} there is $\chi\in\cE(G,st)\cap\Irr(B)$. (Again this
was formulated only for Frobenius maps, but the proof extends verbatim to
Steinberg maps.) Since $B$ is $\si$-stable, so is the $G^*$-class of $s$, and
then the argument in the proof of \cite[Prop.~4.8]{MNST} shows that $\chi$ lies
in an orbit of length~$\ell$ under $\si$. Thus any character of $A$ above $\chi$
has height~1 in~$B_A$.

In the case that $G=\SL_n(\eps q)$ and the defect groups of $B$ are non-cyclic,
using \cite[Prop.~4.11]{MNST} gives that $\chi$ can be chosen to be trivial on
$\bZ(G)_\ell$.
\end{proof}

\subsection{Unipotent blocks}
We now first consider unipotent blocks. For this, recall from \cite[Thm~A]{En00}
that if $F$ is a Frobenius map, any unipotent $\ell$-block of $G$ is labelled
by a unipotent $e$-cuspidal pair $(\bL,\la)$ of $G$ of central $\ell$-defect,
where $e$ is the order of~$q$ modulo~$\ell$, respectively modulo~$4$ if $\ell=2$
(which we denote $e_\ell(q)$). The following strengthens \cite[Lemma~2.8]{KM17}:

\begin{lem}   \label{lem:unipdef}
 Assume $F$ is a Frobenius map.
 Let $(\bL,\la)$ be a unipotent $e$-cuspidal pair of $G=\bG^F$ of central
 $\ell$-defect, where $e=e_\ell(q)$, such that $\ell$ divides
 $|W_{\bG^F}(\bL,\la)|$. Then $W_{\bG^F}(\bL,\la)$ has an irreducible character
 $\phi$ with $\phi(1)_\ell=\ell$, unless one of the following holds:
 \begin{enumerate}[\rm(1)]
  \item We have $\ell=3$, and either $W_{\bG^F}(\bL,\la)\cong\fS_{3c}$,
   $\bG^F=\SL_{3c}(\eps q)$ with $3|(q-\eps)$ and $c\in\{1,2\}$, or
   $\bG^F=E_6(\eps q)$, $W_{\bG^F}(\bL,\la)\cong\fS_3$ and $(\bL,\la)$
   corresponds to Line~$8$ of the $E_6$-tables of \cite[pp.~351,~354]{En00}.
  \item We have $\ell=2$, $W_{\bG^F}(\bL,\la)\cong\fS_2$, and either $\bG$ is of
   classical type, or $\bG$ is of type $E_7$ and $(\bL,\la)$ corresponds to one
   of Lines~$3$ or~$7$ of the $E_7$-table of \cite[p.~354]{En00}.
 \end{enumerate}
\end{lem}

\begin{proof}
First assume $\bG$ is of exceptional type, or that $\bG^F=\tw3D_4(q)$. The
relative Weyl groups $W_{\bG^F}(\bL,\la)$ of unipotent $e$-cuspidal pairs are
listed in \cite[Table~1]{BMM}, and an easy check shows that in the very few
cases when $\ell$ does divide $|W_{\bG^F}(\bL,\la)|$, they do possess characters
$\phi$ with $\phi(1)_\ell=\ell$, unless either $\ell=3$, $\bG$ is of type $E_6$
and we are in case~(1) of the conclusion, or $\ell=2$ and
$W_{\bG^F}(\bL,\la)\cong C_2$ in $\bG$ of type $E_6$, $E_7$, or $E_8$. According
to the tables in \cite[pp.~351,~354,~358]{En00}, the only case with $\la$ of
central $\ell$-defect is in $E_7$ with $\bL$ of type $E_6$ and $\la$ one of the
two cuspidal characters as in~(2).
\par
Next assume that $\bG^F$ is of untwisted type $\tA$. Here, the relative Weyl
groups are wreath products $C_e\wr\fS_a$ for some $a\ge1$ (see
\cite[Exmp.~3.5.29]{GM20}). If $a\ge\ell^2$ then Sylow $\ell$-subgroups
of~$\fS_a$ are non-abelian and the claim follows from the validity of the
EM-conjecture for the principal block of~$\fS_a$ \cite[Thm~2.1]{BM15}.
Assume $a<\ell^2$. If $\ell=2$ and either $a=3$ or $e=2$ then $C_e\wr\fS_a$
has a character as claimed. The case $\ell=2=ae$ is excluded by~(2).
So we have $\ell>2$. Then $e<\ell$ by definition, so if $\ell$ divides
$|W_{\bG^F}(\bL,\la)|$ then $\ell\le a$.  If $2<\ell\le a<\ell^2$ then let
$\mu$ be an $\ell$-core of size $b$, where $\ell\le b<2\ell$ with
$a\equiv b\pmod\ell$. Then any character of height~0 in the $\ell$-block of
$\fS_a$ labelled by $\mu$ has the required property. Such cores exist unless we
are in~(1) of the conclusion, see \cite{GO96}.
If $\bG^F$ is a unitary group, the same argument applies, except that here the
relative Weyl groups have the form $C_d\wr\fS_a$ with $d =e_\ell(-q)$
(see \cite[Exmp.~3.5.29]{GM20}).
\par
For $\bG$ of type $\tB$ or $\tC$, the relative Weyl groups have the form
$C_d\wr\fS_a$, with $d\in\{e,2e\}$ even (see \cite[Exmp.~3.5.29]{GM20}), and by
our above argument, the factor group $\fS_a$ already has a suitable character,
unless we are in one of the exceptions for $\ell=3$, namely $a=3,6$ and $d=2$,
or for $\ell=2$ with $d=2$, $a=1$, excluded by~(2). In case $\ell=3$,
$C_d\wr\fS_a$ has an irreducible character of degree~6. Finally,
for $\bG$ of type $\tD$, the relative Weyl groups occur as normal subgroups of
index at most~2 in relative Weyl groups of type $\tB$-groups, considered
above, and so inherit a character of the required form by Clifford theory.
\end{proof}

Let $S$ be a central quotient of $G$. A block $B$ of $S$ is called
\emph{unipotent} if it is dominated by a unipotent block of $G$. The case $A=S$
of the following result was shown in \cite[Thm~4.4]{FLZ}.

\begin{prop}   \label{prop:unip class}
 Let $\ell$ be a prime and let $S$ be a non-exceptional covering group of a
 finite simple group of classical type in characteristic not~$\ell$. Let
 $A=S\langle\si\rangle$, where $\si$ is an $\ell$-power order field
 automorphism of~$S$. Then $\mh(\tilde B)=1$ for any $\ell$-block $\tilde B$ of
 $A$ of non-abelian defect covering a unipotent block of $S$.
\end{prop}

\begin{proof}
By our assumption, $S$ is a central quotient of a group $G=\bG^F$ as considered
above. Let $B$ denote the unipotent block of $G$ dominating the unipotent block
$B_S$ of $S$ covered by~$\tilde B$. First assume that $\ell=2$ or that $G$ is
of type $\tA_n(\eps q)$ with $\ell|(q-\eps)$. Then the only unipotent block of
$G$ is the principal block by \cite[Thm~A and Prop.~6]{En00}. In this case,
the claim was shown in \cite[Prop.~4.5]{BM15} when $A=S$ and in
\cite[Thm~3.5]{MMR} when $A>S$.
Note that this concludes the discussion of $G=\SL_3(\eps q)$ and $\SL_6(\eps q)$
with $\ell=3|(q-\eps)$ in~(1) of Lemma~\ref{lem:unipdef}, as well as the case
$\ell=2$ in (2) of that result.

So we have $\ell>2$ and $\ell$ does not divide $|\bZ(G)|$. Thus, $\Irr(B_S)$
can be identified with $\Irr(B)$ by inflation and the defect groups of $B_S$ and
of $B$ are isomorphic. If these are abelian, our claim is in
Proposition~\ref{prop:abelian}.

Now consider the case that $B$ has non-abelian defect. Assume $B$ is labelled
by the unipotent $e$-cuspidal pair $(\bL,\la)$, where $e=e_\ell(q)$. Then by
\cite[Lemma~2.9]{KM17} the relative Weyl group $W_G(\bL,\la)$ has order
divisible by~$\ell$. By the above, we may assume we are not in one of the
exceptions of
Proposition~\ref{lem:unipdef}. Then the relative Weyl group $W_G(\bL,\la)$ has
an irreducible character $\phi$ with $\phi(1)_\ell=\ell$. By `preservation of
heights' (see \cite[Prop.~6.5]{KMS}) this implies that the corresponding
unipotent character $\chi\in\Irr(B)$ has height~1. Now unipotent characters
are invariant under all field automorphisms \cite[Thm~4.5.11]{GM20}, so the
deflation of $\chi$ to $S$ extends to a character of height~1 in $\tilde B$,
whence $\mh(\tilde B)=1$.
\end{proof}

\begin{prop}   \label{prop:unip exc}
 Let $\ell$ be a prime. Let $S$ be a non-exceptional covering group of a simple
 group of exceptional Lie type in characteristic not $\ell$ and
 $A=S\langle\si\rangle$, where $\si$ is an $\ell$-power order field
 automorphism of~$S$. Then $\mh(\tilde B)=1$ for any $\ell$-block $\tilde B$
 of $A$ of non-abelian defect covering a unipotent block of $S$.
\end{prop}

\begin{proof}
Let $B_A$ be an $\ell$-block of $A$ with non-abelian defect. Under our
assumptions, $B_A$ is the unique block of~$A$ covering some unipotent block
$B_S$ of $S$. Let $B$ be the unipotent block of~$G$ dominating $B_S$, where
again $S$ is a central quotient of $G=\bG^F$ is as before.

If $B$ is principal, the claim is contained in \cite[Thm~4.4]{BM15} when $A=S$
and in \cite[Thm~3.5]{MMR} when $A>S$. If $B$ has abelian defect, we argue as
follows. If $\ell>2$ the assertion was shown for $S=G$ in
Proposition~\ref{prop:abelian}. The only case for which $S\ne G$ can occur is
for $\bG$ of type $\tE_6$ with $\ell=3$ (because if $\ell{\not|}\ |\bZ(G)|$,
$\Irr(B_S)$ and $\Irr(B)$ are identified as before), but then $G$ has no
3-blocks of abelian defect by the tables in \cite{En00}. The explicit results
in \cite{En00} also show that there is no unipotent $2$-block of $G$ with
abelian defect.

If $G$ is a Suzuki or Ree group, by \cite{KM24} all non-principal unipotent
blocks have abelian defect, so from now on we can assume $F$ is a Frobenius map
and that $B$ has non-abelian defect. In this case, we may argue as in the proof
of Proposition~\ref{prop:unip class} to conclude, unless we are in the
exceptions in~(1) or~(2) of Lemma~\ref{lem:unipdef}, that is, $B$ is the
non-principal unipotent 3-block of $G=\tE_6(\eps q)_\SC$ with $3|(q-\eps)$ or
one of the non-principal 2-blocks of $G=E_7(q)_\SC$ corresponding to one of
Lines~$3$ or~$7$ of the $E_7$-table of \cite[p.~354]{En00}.

First, assume $G=\tE_6(q)_\SC$ with $3|(q-1)$ and $B$ is the non-principal
unipotent 3-block labelled by the $1$-cuspidal pair $(\bL,\la)$, where $\bL$ is
the standard Levi
subgroup of type $D_4$. By \cite[p.~353]{En00} a defect group $D$ of
$B$ is an extension of the homocyclic group $\bZ(\bL^F)_3$ of order
$(q-1)_3^2$ by a cyclic group of order~3 inside the normaliser $\bN_G(\bL)$.
Clearly, $D$ may be chosen to be normalised by the field automorphism
$\si$ of $3$-power order. Let $s\in G^*$ be an element of order~3 with
centraliser of type~$D_5$, and $\chi\in\cE(G,s)$ in Jordan correspondence with
a unipotent character of $C_{G^*}(s)$ in the 1-Harish-Chandra series of $\la$.
Then $\chi\in \Irr(B)$ by \cite[Thm~B]{En00}, and $\chi$ is of height~1 (see
\cite[Lemma~2.14(a)]{KM17}). We now distinguish two cases. If $9|(q-1)$ then
$s$ is a third power, so contained in $[G^*,G^*]$, hence $\chi$ deflates to a
character of $G/\bZ(G)$. As there is a unique $G^*$-class of elements $s$ as
above, $\chi$ is invariant under all field automorphisms of $G$, hence of $S$,
and so extends to a character of height~1 in $B_A$. If $3||(q-1)$ then $q$
cannot be a third power and thus $A=S$. In this case, $\chi$ is a character of
height~1 in $B$ while the corresponding block of $G/\bZ(G)$ has abelian defect.
Since these are the only possibilities for $A=S$ in this situation, we are
again done. The argument for $G=\tE_6(-q)_\SC$ is entirely similar.

Finally, assume $G=\tE_7(q)_\SC$ with $4|(q-1)$ and $\hat B$ is the
non-principal unipotent 2-block labelled by $(\bL,\la)$ where $\bL$ is the
standard Levi subgroup of type $E_6$. Here by \cite[p.~355]{En00} a defect
group $D$ of $B$ is dihedral of order $2(q-1)_2$. Let $s\in G^*$ be an element
of order~4 with centraliser of type $E_6$, and $\chi\in\cE(G,s)$ in
Jordan correspondence with a unipotent character in the 1-Harish-Chandra series
of $\la$. Then $\chi\in B$ by \cite[Thm~B]{En00}, and $\chi$ is of height~1,
see \cite[Lemma~2.16(a)]{KM17}. We again distinguish two cases. If $8|(q-1)$
then $s$ is a square, so contained in $[G^*,G^*]$, hence $\chi$ deflates to a
$\si$-stable character of $G/\bZ(G)$, giving rise to a character of height~1 in
$B_A$. If $4||(q-1)$ then $q$ cannot be a square, and thus $A=S$. In this
case, $\chi$ is a character of height~1 in $B$ while the corresponding block
of $G/\bZ(G)$ has abelian defect. This covers all possible groups $A$ in this
case and we are done. Again, the case when $4|(q+1)$ is entirely similar.
\end{proof}

\subsection{Type $\tA$}
Here we address the groups of types $\tA$ and $\tw{2}\tA$.

\begin{prop}   \label{prop:typeA}
 Let $\ell\geq 3$ and $A$ such that $S=\bF^\ast(A)$ is a central quotient of
 $\SL_n(\eps q)$ and $A$ is generated over $S$ by outer automorphisms of
 $\ell$-power order. Let $B$ be an $\ell$-block of $A$ with non-abelian defect
 groups. Then $\mh(B)=1$.
\end{prop}

\begin{proof}
Let $\wt{G}:=\GL_n(\eps q)$, $G:=\SL_n(\eps q)=\bG^F$ and $S=G/Z$ with
$Z\leq \bZ({\wt{G}})$. We have $A\leq\wt{G}/Z\rtimes\langle\si\rangle$, where
$\si$ is an $\ell$-power order field automorphism of $\wt{G}$ inducing an
automorphism on $S$. 

Let $G_\ell$ be the group $G\leq G_\ell\leq \wt{G}$ such that $G_\ell/G$ is the
Sylow $\ell$-subgroup of $\wt{G}/G$ (so $|\wt{G}/G_\ell|$ is $\ell'$ and
$G_\ell/G$ is an $\ell$-group). Write $\wt{S}:=\wt{G}/Z$ and $S_\ell:=G_\ell/Z$.
Now, define $X:=A\cap \wt{S}$ and $X_\ell:=A\cap S_\ell$. We have $A/X$ is
cyclic of $\ell$-power order and $X/X_\ell$ is cyclic of order prime to $\ell$.
Further, there is $\delta\in \wt{S}$ such that the image of
$\al:=\delta\si\in A/S$ in $A/X$ generates $A/X$.
Let $B_X$ be a block of $X$ covered by $B$ and $B_S$ a block of $S$ covered by
$B_X$ (and hence by~$B$). Let $B_{\wt{S}}$ be a block of $\wt{S}$ covering $B_X$
(and hence $B_S$). Further, let $B_G$ be the block of $G$ dominating $B_S$, and
$\wt{B}$ the block of $\wt{G}$ dominating $B_{\wt{S}}$, so that $\wt{B}$ covers
$B_G$.

Let $I(B_X)$ be the stabiliser of $B_X$ in $A$. By the Fong--Reynolds theorem
\cite[Thm~9.14]{N98}, induction of characters $\Ind_{I(B_X)}^A$ from the
block of $I(B_X)$ above $B_X$ to $A$ gives all characters in the block $B$ and
moreover preserves heights. Then we may without loss assume that the block
$B_X$ is $\al$-invariant. Note then that the unique block $B_{XS_\ell}$ of
$XS_\ell$ covering $B_X$ must also be $\al$-invariant. 

Let $\wt{s}\in\wt{G}^\ast$ and $s\in G^\ast$ be semisimple $\ell'$-elements
such that $\Irr(\wt{B})\subseteq\cE_\ell(\wt{G},\wt{s})$ and
$\Irr(B_G)\subseteq \cE_\ell({G}, {s})$. For each group $G\leq H\leq \wt{G}$,
we define $\cE(H,\wt{s})$ to be the set of constituents of restrictions of
characters of $\cE(\wt{G},\wt{s})$ to $H$, as in \cite[Def.~1.1]{CE99}. (Note
that with this notation, $\cE(G,\wt{s})$ is the same as $\cE({G},{s})$.)

Set $C:=C_{\wt{G}^\ast}(\wt{s})$. Note that $C$ is a direct product
$C=\prod C_i$ of groups $C_i$ of the form $\GL_m(\eta q^d)$ for suitable
$\eta\in\{\pm1\}$ and positive integers $m$ and~$d$. By \cite[Thm~(7A)]{FS82},
Jordan decomposition maps $\Irr(\wt B)\cap \cE(\wt{G},\wt{s})$ to
$\Irr(B_C)\cap\cE(C,1)$, where $B_C=\prod B_{C_i}$ is some unipotent block of
$C=\prod C_i$ with defect groups isomorphic to a defect group $\wt{D}$ of
$\wt{B}$. We further see that this map is height-preserving, and by
\cite[Thm~3.1]{CS13} can even be chosen to be $\si$-equivariant.
\medskip

(i) First, suppose that $\wt{D}$ is abelian. If the block $B_{X}$ has cyclic
defect groups, then by Proposition~\ref{prop:meta}, there is a height-zero
character of $B_X$ in an orbit of length $\ell$ under the action of $A$, giving
a character of height one in the (unique) block $B$ of $A$ covering $B_X$. 
So, we assume the defect groups of $B_X$ are not cyclic, and hence $\wt{D}$ is
not cyclic.

By Proposition \ref{prop:abelian} and its proof, there is some $\ell$-element
$t\in C_{G^\ast}(s)$ and $\chi\in\cE(G, st)$ such that $\chi$ and $\cE(G,st)$
lie in orbits of size~$\ell$ under $\si$. Notice that since the class of $st$
is stable under diagonal automorphisms, we further have $\cE(G,st)$ has an
orbit of size $\ell$ under $\al$. If $D_G:=\wt{D}\cap G$ is
non-cyclic, Proposition~\ref{prop:abelian} shows that $\chi$ can be chosen
trivial on $\bZ(G)_\ell$. Further, the construction in \cite[Thm~4.11]{MNST}
yields a preimage $\wt{s}\wt{t}\in\wt{G}^\ast$ of $st$ such that
$\cE(\wt{G},\wt{s}\wt{t})$ has an orbit of size $\ell$ under~$\si$, and
applying \cite[Thm~4.1]{Ma14} we obtain
$\wt\chi\in\Irr(\wt{B}\mid\chi)\cap\cE(\wt{G},\wt{s}\wt{t})$ in an orbit of
size $\ell$ under~$\si$. Hence $\wt\chi$ is also in an orbit of size $\ell$
under $\al$.  When $D_G$ is cyclic but $\wt{D}$ is not, the
constructions in \cite[Thm~4.11]{MNST} still work. 

Then since $\wt{S}/XS_\ell$ has order prime to $\ell$, we conclude that there
is some $\psi\in\Irr(B_{XS_\ell}\mid \wt{\chi})$ that is also invariant under
$\al^\ell$.  Note that the defect groups of $B_{XS_\ell}$ and of the unique
block $B_{S_\ell}$ of $S_\ell$ above $B_S$ have size $|\wt{D}/Z_\ell|$, where
$Z_\ell$ is the Sylow $\ell$-subgroup of $Z$. Those of the unique block
$B_{X_\ell}$ of $X_\ell$ above $B_S$ have size $|(\wt{D}/Z_\ell)\cap X_\ell|$.
Since $X/X_\ell$ is prime to $\ell$, we have the defect groups of $B_X$ have
size $|(\wt{D}/Z_\ell)\cap X_\ell|=|(\wt{D}/Z_\ell)\cap X|$ as well. Since
restriction of characters is a bijection from $\cE(G_\ell, \wt{s})$ to
$\cE(G,s)$ by \cite[Prop.~1.3]{CE99}, we see $B_{S}$ and $B_{X_\ell}$ are 
$S_\ell$-invariant. Hence
$[S_\ell:\wt{D}/Z_\ell]_\ell=[S:D_G/Z_\ell]_\ell
 =[X_\ell:(\wt{D}/Z_\ell)\cap X_\ell]_\ell$ by Lemma \ref{lem:pfactor}. Since
$\wt{D}$ is abelian, note that the characters of each of these blocks are of
height zero, by the ``if" direction of Brauer's height zero conjecture
\cite{KM13}. Then for any $\vhi\in\Irr(B_X\mid \psi)$, we must have
\[\vhi(1)_\ell=[X:(\wt{D}/Z_\ell)\cap X_\ell]_\ell
  =[X_\ell:(\wt{D}/Z_\ell)\cap X_\ell]_\ell=[S_\ell:\wt{D}/Z_\ell]_\ell
  =[XS_\ell:\wt{D}/Z_\ell]_\ell=\psi(1)_\ell\]
Since $XS_\ell/X$ is an $\ell$-group, it follows that $\psi$ must restrict
irreducibly to~$\vhi$, and therefore $\vhi$ is also $\al^\ell$-invariant. 

We now claim that $\vhi$ is not $\al$-invariant. Note that
$\vhi\in\cE(X,\wt{s}\wt{t})$. If $\vhi^\al=\vhi$, then $\vhi$ lies above
the deflations  of characters in  $\cE(G, st)$ and of characters in
$\cE(G,st)^\al\neq \cE(G,st)$. Then the same is true for $\wt\chi$, a
contradiction. So we see $\vhi$ must lie in an orbit of size $\ell$ under the
cyclic group $A/X$, and hence there is a character of $B$ above $\vhi$ with
height one.
\medskip

(ii) From now on, we suppose that $\wt{D}$ is non-abelian. First, suppose
either $\ell\geq 5$, or $\ell=3$ and not all factors $C_i$ of~$C$ with
non-abelian defect groups for the corresponding block $B_{C_i}$ are of the form
$\GL_{3c}(\eta q^d)$ with $c\in\{1,2\}$ and $3\mid (q^d-\eta)$. (That is, we
assume here that at least one of the blocks whose tensor product makes up $B_C$
has non-abelian defect groups and is not of a type singled out in
Lemma~\ref{lem:unipdef}.) Then as in the proof of
Proposition~\ref{prop:unip class}, there
is a unipotent character of height one in $B_C$, and hence a member $\wt\chi$
of $\Irr(\wt{B})\cap\cE(\wt{G}, \wt{s})$ $=\cE(\wt{G}, \wt{s})$ of height one.
Note that $\wt\chi$ is trivial on $\bZ(\wt{G})_\ell$, by \cite[Prop.~1.2]{CE99}
and hence on $Z$ by \cite[Thm~9.9(c)]{N98}.

Let $B_{G_\ell}$ be the block of $G_\ell$ above $B_G$, covered by $\wt{B}$.
Then $\Irr(B_{G_\ell})\cap\cE(G_\ell, \wt{s})$ contains some
$\chi'\in\Irr(G_\ell\mid \wt\chi)$, also with height one using
Lemma~\ref{lem:Op'}.
By \cite[Prop.~1.3]{CE99}, restriction of characters is a bijection from
$\cE(G_\ell, \wt{s})$ to $\cE(G,s)$, so $\chi:=\chi'|_G$ lies in
$\Irr(B_G)\cap\cE(G,s)$ and has height one using Lemma \ref{lem:pfactor}.
Recall further that $\wt\chi$ is trivial on $Z$, and hence so are $\chi$ and
$\chi'$. We continue to denote by $\chi$ and $\chi'$ the deflations to
$S_\ell=G_\ell/Z$ and $S=G/Z$.

Let $B_{X_\ell}$ be the unique block of $X_\ell$ covering $B_S$, so it follows
that $\vhi:=\chi'|_{X_\ell}\in\Irr(B_{X_\ell})$. Let $\wt{\vhi}\in\Irr(B_X)$
lie above $\vhi$. Notice that $\wt\vhi$ also has height one in $B_X$, using
Lemma~\ref{lem:Op'}. Recalling that $B_{XS_\ell}$ is $\al$-invariant, we have
$\al$, and hence $\si$, permutes the blocks of $\wt{S}$ above $B_{XS_\ell}$.
Then $\si$ permutes the corresponding blocks of $\wt{G}$, which lie above
$B_{G_\ell}$ and include~$\wt{B}$. Now, each of these blocks of $\wt{G}$ lies
in $\cE_\ell(\wt{G},\wt{s}')$ for some semisimple $\ell'$-element
$\wt{s}'\in\wt{G}^\ast$ such that $\cE(\wt{G}, \wt{s}')$ contains characters
whose restrictions to $G_\ell$ have constituents in $\cE(G_\ell, \wt{s})$.
These in turn are of the form $\wt{s}'=\wt{s}z$ with $z\in\bZ(\wt{G}^\ast)$
corresponding by duality to a character $\hat z$ of $\wt{G}/G$ with $G_\ell$ in
its kernel, so that $\hat z$ is a character of $\wt{G}/G_\ell$. So, $\si$ acts
on the set of classes of such $\wt{s}z$. Since the number of these is prime
to~$\ell$, it follows that one of the series $\cE(\wt{G}, \wt{s}')$ with
$\wt{s}'=\wt{s}z$ must be $\si$-stable. Replacing $\wt{s}$ by this~$\wt{s}'$,
we may therefore further assume that $\cE(\wt{G}, \wt{s})$ is $\si$-stable.
By \cite[Thm~4.1]{Ma14}, this means each character in $\cE(\wt{G},\wt{s})$ is
$\si$-stable. Hence, we may assume our $\wt{\chi}$ is stable under $\si$ and
therefore $\al$.

Since $\ell$ does not divide the order of the cyclic group $\wt{G}/G_\ell$,
there must be some such~$\chi'$ as before that is also $\al$-invariant, so
$\vhi$ is $\al$-invariant. But the cyclic group $X/X_\ell$ also has order prime
to $\ell$, so the number of characters in $\Irr(X|\wt{\chi})$ above $\vhi$ is
prime to $\ell$, and we may assume $\wt\vhi$ is $\al$-invariant. Then any
character of $A$ above $\wt\vhi$ has height one, as desired.
\medskip

(iii) Now suppose that $\ell=3$, $\wt{D}$ is still non-abelian, and the factors
of $C$ with non-abelian defect group are all of the form
$\GL_3(\eta q^d)$ or $\GL_6(\eta q^d)$ with $3\mid (q^d-\eta)$, and $n>3$.
Write $C=C_1\times C_2$, where $C_1$ is one of the factors $\GL_{3c}(\eta q^d)$
with $c\in\{1,2\}$ and $C_2$ is the product of the remaining factors. 

Let $\wt{t}$ be the image in $\wt{G}$ under the natural embedding of
$\diag(\ze,\ze^{-1},I_{3c-2},I_{n-3c})\in C$, where $\ze\in\FF_{q^{2d}}^\times$
with $|\ze|=3$. Then $\wt t$ is $\si$-invariant,
and clearly is not $\wt{G}^\ast$-conjugate to $\wt{t}z$ for any
$1\neq z\in\bZ(\wt{G}^\ast)$. Further, we have $[C:C_{C}(\wt{t})]_3=3$.
Then the corresponding semisimple character in $\cE(C,\wt t)$ in (the principal
block) $B_C$ has height one and is $\si$-invariant. Arguing as in (ii), we may
assume $\cE(\wt{G},\wt{s})$, and hence $\wt{B}$ and the class of $\wt{s}$ are
$\si$-stable, so $\al$-stable. Then the semisimple character
$\chi_{\wt{s}\wt t}$ of $\wt{G}$ lies in $\wt{B}$, has height one, and is
$\al$-invariant, using \cite[Thm~(7A)]{FS82} and \cite[Thm~3.1]{CS13}.
Since $\wt t\in[\wt{G}^\ast,\wt{G}^\ast]$ and the characters in
$\cE(\wt{G},\wt{s})$ are trivial on $Z$ as above, we have by
\cite[Lemma~2.2]{Ma07} and its proof that $\chi_{\wt{s}\wt{t}}\in\Irr(\wt{G})$
is also trivial on $Z$. Let $B_{G_\ell}$ be the block of $G_\ell$ above $B_G$,
covered by $\wt{B}$. Then $\Irr(B_{G_\ell})\cap\cE(G_\ell, \wt{s}\wt{t})$
contains some $\chi'\in\Irr(G_\ell\mid \wt\chi)$, also with height one using
Lemma~\ref{lem:Op'} and $\al$-invariant since $\wt{G}/G_\ell$ is a
$3'$-group.    

Further, by Clifford theory and the fact that $G_\ell/G$ is cyclic, the number
of irreducible constituents of $\chi'$ on restriction to $G$ is the number of
$\beta\in\Irr(G_\ell/G)$ such that $\chi'\beta=\chi'$. 
But any such $\beta$ corresponds to an element $z\in\bZ(\wt{G}^\ast)$ of 3-power
order satisfying $\wt{s}\wt{t}z$ is conjugate to $\wt{s}\wt{t}$, so the only
possibility is $z=1$, whence $\beta=1$. That is, $\chi'$ restricts irreducibly
to a character in $B_G\cap\cE(G, st)$ with height one, where $t$ is the image of
$\wt{t}$ in~$G^\ast$. From here, we
obtain a character in $B$ with height one from the same argument as (ii).

(iv) Finally, suppose that $\wt{G}=\GL_3(\eps q)=C$ with $\ell=3\mid (q-\eps)$.
Note that in this case, field automorphisms of $3$-power order centralize the
outer diagonal automorphisms, so $A/S$ is generated by a field
automorphism~$\si$. Consider the element $\wt{t}=\diag(\zeta,\zeta,1)$ with
$|\zeta|=3$. Since $\si$ is a field automorphism of 3-power order, it is
induced by a Frobenius morphism $F_{q_0}$, where $q=q_0^{3^a}$ for some
positive integer $a$ and
power $q_0$ of the defining characteristic. Then $\wt{t}$ is $\si$-stable since
$3\mid (q_0-\eps)$, by Fermat's little theorem. If $9\mid (q-\eps)$, then there
is $z\in\bZ(\wt{G}^\ast)$ of order~9 such that
$\wt{t}z\in[\wt{G}^\ast, \wt{G}^\ast]$, and hence the image $t$ of $\wt{t}$ in
$G^\ast$ lies in $[G^\ast, G^\ast]$. In particular, the characters of $G$ lying
below the corresponding semisimple character $\chi_{\wt{s}\wt{t}}$ of $\wt{G}$
are trivial on $\bZ(G)$. If instead $3\mid\mid(q-\eps)$, and $B_S$ has
non-abelian defect, then we have $S=G$. Further, $\wt{t}$ is not
$\wt{G}^\ast$-conjugate to~$\wt{t}z$ for any $1\neq z\in\bZ(\wt{G}^\ast)$. Then
$\chi_{\wt{s}\wt{t}}\in\Irr(\wt{G})$ lies in $\wt{B}$, has height one, and is
fixed by~$\si$. In either case, the previous arguments now yield a character
of~$B$ with height one.
Finally, if $3\mid\mid(q-\eps)$ and $B_S$ has abelian defect groups while
$\wt{D}$ is non-abelian, then $S=G/\bZ(G)=\PSL_3(\eps q)$ and $B_S$
is the principal block, and the result is part of \cite[Thm~3.5]{MMR}.

We remark that the character tables of $\SL_3(\eps q)$ and $\GL_3(\eps q)$ are
known explicitly and the claim could thus also be checked by inspection.
\end{proof}

\begin{prop}   \label{prop:typeA2}
 Let $\ell=2$ and let $G=\SL_n(\eps q)$. Let $B$ be a quasi-isolated $2$-block
 of $G$ with non-abelian defect groups. Then $\mh(B)=1$.
\end{prop}

\begin{proof}
Let $s\in G^*=\PGL_n(\eps q)$ be a quasi-isolated element of odd order $d|n$
such that $\Irr(B)\subseteq\cE_2(G,s)$. Note that by \cite[Thm~21.14]{CE04}
in conjunction with \cite{BDR} we in fact have $\Irr(B)=\cE_2(G,s)$, and a
defect group of $B$ is a Sylow 2-subgroup of $C_{\bG^*}(s)$ (as it has the
same order by \cite[Lemma~2.6]{KM13}). Let $\tilde s\in\tiG^*=\GL_n(\eps q)$ be
a preimage of~$s$. Then $C_{\tiG^*}(\tilde s)=\GL_{n/d}(q^a)^{d/a}$ for some
$a|d$ (see \cite{Bo05}). If $n/d\ge3$ then $\GL_{n/d}(q^a)$ has a unipotent
character of height~1 by Lemma~\ref{lem:unipdef}, and via Jordan decomposition
this gives rise to a character in $\cE(G,s)$ and hence in $B$ of height~1, as
$d$ is odd. So assume $n/d\le2$. If $n/d=1$ then $B$ has abelian defect. So
we are left with the case $C_{\tiG^*}(\tilde s)=\GL_2(q^a)^{d/a}$ and $n=2d$.
Here the principal block of $C_{\tiG^*}(\tilde s)$ contains characters of
height~1 with $\bZ(\tiG^*)$ in their kernel,  which give rise to height~1
characters of $C_{G^*}(s)$ and via Jordan decomposition of $B$.
\end{proof}

\subsection{Quasi-isolated blocks}
We now turn to \emph{quasi-isolated blocks} of groups of Lie type, that is,
blocks labelled by quasi-isolated elements in the dual group.

\begin{prop}   \label{prop:isolatedclassicalnofield}
 Let $\ell$ be a prime and let $G=\bG^F$, where $\bG$ is a simple simply
 connected group of classical type, and $F:\bG\rightarrow\bG$ is a Frobenius
 endomorphism defining $G$ over $\FF_q$ with $q$ a power of $p\ne\ell$. Then
 $\mh(B)=1$ for all quasi-isolated $\ell$-blocks $B$ of $G$ with non-abelian
 defect.
\end{prop}

\begin{proof}
Let $B$ be an $\ell$-block of $G$ with a non-abelian defect group $D$. We have
$\Irr(B)\subseteq \cE_\ell(G,s)$ for a semisimple $\ell'$-element $s\in G^*$.
By Proposition~\ref{prop:unip class} we may assume that $B$ is not a unipotent
block, so $s\neq 1$. We may further assume that $\bG$ is not of type
$\tA_{n-1}$, by Propositions~\ref{prop:typeA} and~\ref{prop:typeA2}. Then since
the remaining groups of classical types have no non-trivial quasi-isolated
$2'$-elements, we have that $\ell$ is odd.

Let $\iota\colon\bG\rightarrow\wt\bG$ be a regular embedding, and write
$\wt{G}:=\wt\bG^F$. Note that by our assumptions $\wt{G}/G\bZ(\wt{G})$ and
$G\cap\bZ(\wt{G})$ are $\ell'$-groups.  Let $\wt{s}\in\wt{G}^\ast$ be a
semisimple $\ell'$-element such
that $\iota^\ast(\wt{s})=s$. Then there is a block $\wt{B}$ of $\wt{G}$
covering $B$ with $\Irr(\wt B)\subseteq \cE_\ell(\wt G, \wt s)$, and by
Lemmas~\ref{lem:Op'} and~\ref{lem:centralprod}, it suffices to prove the
statement for $\wt{B}$.

Let $\bH^\ast:=C_{\wt\bG^\ast}(\wt{s})$ and let $\bH$ be dual to
$\bH^\ast$. Writing $H:=\bH^F$, since $\ell$ is odd we may argue exactly as in
the first half of the proof of \cite[Prop.~3.17]{MNST} to obtain that there is
a unipotent block $b$ of $H$ such that $b$ and $\wt{B}$ have isomorphic defect
groups and the unipotent characters in $b$ are mapped injectively through
Jordan decomposition to $\Irr(\wt{B})\cap \cE(\wt{G},\wt{s})$. Further, note
that Jordan decomposition preserves heights. 

Let $\bH_0:=[\bH, \bH]_\SC$ and $H_0:=\bH_0^F$. By \cite[Thm~(i)]{CE94}, the
block partition of unipotent characters is independent of the isogeny type,
and hence there is a unipotent block $b'$ of~$H_0$ whose unipotent characters
correspond to those in $b$, with the same degrees. 
We have $H_0$ is the direct product of groups of the form $K_i:=\bK_i^{F_i}$
where each $\bK_i$ is simple of simply connected type and $F_i$ is some power
of $F$. Let the block $b'$ of $H_0$ be the product of the blocks $c_i$ of $K_i$.

If $\ell\neq 3$ or at least one component of $H_0$
is not of the types in Lemma~\ref{lem:unipdef}(1), then since $b'$ has
non-abelian defect groups, the proof of Proposition~\ref{prop:unip class},
together with \cite[Lemma~2.11]{KM17}, gives that there is a unipotent
character in $b'$ with height~1. Then \cite[Thm~(i)]{CE94} implies that the
corresponding unipotent character of $b$ also has height~1, and thus its Jordan
correspondent in $\tilde B$ is as required.

Now, suppose that $\ell=3$ and $H_0$ has only components of the types in
Lemma~\ref{lem:unipdef}(1) or with $c_i$ of abelian defect. From our assumption,
note that $G$ is not $\SL_n(\eps q)$ in this case. We see from the
classification of quasi-isolated elements \cite[Tab.~II]{Bo05} that when this
occurs, $H_0$ has only one component $K:=K_i$ such that $c:=c_i$ has
non-abelian defect groups. 

Note that $\SL_3(\eps q)$ and $\SL_6(\eps q)$ have a unique unipotent block for
$\ell=3\mid (q-\eps)$. Then $c$ is the principal $3$-block of~$K$. Further,
from the proof of \cite[Prop.~3.17]{MNST}, we see that given an $\ell$-element
$t\in\bH^{*F}$, we have $\wt\chi\in\cE(\wt{G}, \wt{s}t)\cap\Irr(\wt{B})$ if and
only if $\psi\in\cE(H, t)\cap\Irr(b)$, where $\chi$ and $\psi$ map to the same
unipotent character of $\bH^{*F}$ under Jordan decomposition. Note that this
correspondence also preserves heights. From here, we obtain a height-one
character in $H$ as in part~(iv) of the proof of
Proposition~\ref{prop:typeA}. 
\end{proof}

Let $Z\le\bZ(G)$, then as above we say a block of $G/Z$ is quasi-isolated if
any block of $G$ dominating it is quasi-isolated.

\begin{thm}   \label{thm:isolated}
 Let $S$ be a covering group of a simple group of Lie type in characteristic
 $p\ne\ell$, $\si$ an $\ell$-power order field automorphism of~$S$ and
 $A:=S\langle\si\rangle$. Then $\mh(B_A)=1$ for all $\ell$-blocks $B_A$ of $A$
 with non-abelian defect covering quasi-isolated blocks of $S$.
\end{thm}

\begin{proof}
If $S$ is an exceptional covering group of $S/\bZ(S)$, then this was shown in
Proposition~\ref{prop:exc cov}. The case $S=\tw2F_4(2)'$ was disposed of in
Proposition~\ref{prop:spor}. Hence, we may exclude these situations, so that
there is a simple simply connected group $\bG$ with Steinberg map $F:\bG\to\bG$
such that $S$ is a central quotient of $G=\bG^F$. Let $B_A$ be an $\ell$-block
of $A$ with non-abelian defect and let $B_S$ the block of $S$ covered by $B_A$.
If $B_S$ is unipotent, our claim is shown in Propositions~\ref{prop:unip class}
and~\ref{prop:unip exc}. So now assume $B_S$ is a quasi-isolated non-unipotent
$\ell$-block of $S$. Thus $B_S$ is dominated by a quasi-isolated block $B$
of~$G$, and $\Irr(B)$ lies in the union of Lusztig series $\cE_\ell(G,s)$ for
some quasi-isolated $\ell'$-element $1\ne s\in G^*$.  Applying the
Fong--Reynolds theorem \cite[Thm~9.14]{N98}, we assume that $B$ is
$A$-invariant, and by Proposition~\ref{prop:typeA}, we may again assume that
$\bG$ is not of type $\tA$.

(i) First, assume $\bG$ is of classical type. Then the only quasi-isolated
elements are of 2-power order (see e.g.\ \cite[Tab.~2]{Ma17}).
These give rise to quasi-isolated blocks for primes $\ell>2$. As $|\bZ(G)|$ is
a 2-power, any $\ell$-block of $S$ is an $\ell$-block of $G$ with isomorphic
defect groups, so we may assume $S=G$. Further, in this case, we may assume
$A\neq S$, by Proposition~\ref{prop:isolatedclassicalnofield}, and that $B$ has
non-abelian defect groups by Proposition~\ref{prop:abelian}. 

Keep the notation of the proof of
Proposition~\ref{prop:isolatedclassicalnofield}. We lift $\si$ to a field
automorphism of~$G$ and then extend it to $\wt{G}$. Let $\wt{B}$ be an
$\ell$-block of $\wt{G}$ above~$B$. From the description of isolated elements in
\cite{Bo05}, we see the conjugacy class of $\wt{s}$ is $\si$-stable. Since $B$
has non-abelian defect groups, note that $\wt{B}$, and hence the unipotent
block $b$ of $H$ from Proposition~\ref{prop:isolatedclassicalnofield}, also has
non-abelian defect groups. Now, the unipotent $\ell$-blocks of any component
of~$H_0$ of type $\tA_1$ have abelian defect groups. Hence our assumption that
$B$ has non-abelian defect ensures that not all components of $\bH^\ast$ can be
of type $\tA_1$ (hence, nor of type $\tD_2$).
Since $\bG$ is type $\tB$, $\tC$, or $\tD$, we see from \cite{Bo05} that
if $\bH^\ast$ has more than two components that are isomorphic to one another
(allowing additional components not isomorphic to those), then those are of
type $\tA_1$ and there is one additional component, say $K$, that is not
isomorphic to these. Note then that the corresponding block of $K$ has
non-abelian defect groups, again from our assumption that $B$ has non-abelian
defect groups and is stable under $\si$. If at most two components are
isomorphic, then $\si$ preserves all components since it has odd order.

Now, by \cite[Thm~3.1]{CS13}, there is a $\si$-equivariant Jordan decomposition
for~$\wt{G}$. Since we are in the cases of $\tB, \tC,$ and $\tD$, the
unipotent characters of $\bH^{*F}$ are stable under field automorphisms (see
\cite[Thm~4.5.11]{GM20}), so we see the height-one character in $\wt{B}$, say
$\wt\chi$, from before can be chosen to be also $\si$-invariant.

Now, since $[\wt{G}:G\bZ(\wt{G})]$ is $\ell'$, we see there must exist a
character of the corresponding block of $G\bZ({\wt{G}})$ above $B$ invariant
under~$\si$. Then the irreducible restriction $\chi$ of that character in $B_S$
is invariant under $\si$ as well. Hence the extension of $\chi$ to $A$ lies in
$B_A$ and has height one, yielding $\mh(B_A)=1$.

(ii) Now assume $S$ is of exceptional type. In \cite[Prop.~4.1]{BM15} we had
shown the claim for $S$ of rank less than~4 in the case when $A=S$.
We first consider the case that $S=G$ or $\ell$ does not divide $|\bZ(G)|$, so
that $B_S$ has the same invariants as $B$.

If the defect groups of $B$ are abelian, we conclude by
Proposition~\ref{prop:abelian} when $\ell\ge3$, while, by inspection of the
relevant tables in \cite{KM13} all non-unipotent quasi-isolated 2-blocks of $S$
have non-abelian defect. Now note that the Sylow $\ell$-subgroups of $G$
(and hence any defect groups) can only be non-abelian if $\ell$ divides the
order of the Weyl group of $G$ (see \cite[Thm~25.14]{MT}), which we now assume.
If $G$ is a Suzuki or Ree group, this only leaves the case $G=\tw2F_4(q^2)$
with $q=2^{2f+1}$ and $\ell=3$. But this group has no quasi-isolated elements
of order larger than~3.

So we may assume $F$ is a Frobenius map. Then our assumptions imply in
particular that $e=e_\ell(q)\in\{1,2,3,4,6\}$. Let $s\in G^*$ be an
isolated $\ell'$-element such that $\Irr(B)\subset\cE_\ell(G,s)$. Then by
\cite[Thm~1.2]{KM13} and \cite[Thm~1.1]{Ho22} the block $B$ is labelled by an
$e$-cuspidal pair $(\bL,\la)$ of quasi-central $\ell$-defect and
$\Irr(B)\cap\cE(G,s)$ is a union of $e$-Harish-Chandra series which satisfy an
$e$-Harish-Chandra theory. Furthermore, the defect groups of $B$ are
non-abelian only if $\ell$ divides the order of the relative Weyl group
$W_G(\bL,\la)$. The $W_G(\bL,\la)$ can be found in the tables of
\cite{KM13,Ho22} and \cite[\S6.2]{KM23}. It transpires that $|W_G(\bL,\la)|$ is
prime to~7 when $e\in\{3,6\}$, and it is prime to~5 when $e=4$ except when
$C_{G^*}(s)\cong A_4(q^2)$ for $G=\tE_8(q)$, but in this latter case $s$ is a
5-element, contrary to assumption. So now we have $e\in\{1,2\}$.
Going through the list of cases, we see that whenever $|W_G(\bL,\la)|$ is
divisible by~$\ell\in\{3,5,7\}$, then $W_G(\bL,\la)$ possesses a character of
degree divisible by~$\ell$ exactly once. (Note here that the first three lines
of Table~5 in \cite{Ho22} contain several obvious misprints.) By
\cite[Prop.~6.5]{KMS} the corresponding
character $\chi\in\cE(G,s)\cap\Irr(B)$ also has height~1.  For $\ell=2$ the
same argument applies except for the blocks listed in Table~\ref{tab:ell=2}
(the cases with $e_2(q)=1$) and their Ennola duals, which all arise for
elements $s$ of order~3, with $3$ dividing $q+1$.

\begin{table}[htb]
\caption{Some isolated $2$-blocks in exceptional groups, $e_2(q)=1$}   \label{tab:ell=2}
$$\begin{array}{c|rccccc}
 G& C_{G^*}(s)& \bL^F& C_{\bL^*}(s)^F& \la& W_G(\bL,\la)\\
\hline
 \tF_4(q)& \tw2\tA_2(q)^2& \Phi_1^2\tA_1(q)^2& \Phi_1^2\Phi_2^2& 1& A_1\times A_1\\
 \tw2\tE_6(q)& \tw2\tA_2(q)^3.3& \Phi_1^3\Phi_2^2\tA_1(q)& \Phi_1^3\Phi_2^3& 1& A_1\wr3\\
       & \tw2\tA_2(q^3).3& \Phi_1\tA_2(q^2)\tA_1(q)& \Phi_1\Phi_2\Phi_3\Phi_6.3& 1& A_1\\
 \tE_8(q)& \tw2\tE_6(q).\tw2\tA_2(q)& \Phi_1\tE_7(q)& \Phi_1\Phi_2\tw2\tE_6(q)& \tw2E_6[\theta^{\pm1}]& A_1\\
\end{array}$$
\end{table}

We discuss the four types of blocks from Table~\ref{tab:ell=2}. For this, we
note that
when $q\equiv1\pmod4$ then any group isogenous to $\SU_3(q)$ possesses a unique
class of elements of order~4, lying in maximal tori of order $q^2-1$. In
particular, this class is invariant under all automorphisms. So in all cases we
can choose an element $t\in C_{G^*}(s)$ of order~4 which is as above in one of
the $\tw2A_2$-factors of $C_{G^*}(s)$ and trivial in all other factors of
$C_{G^*}(s)$ and whose conjugacy class is invariant under all automorphisms
of $C_{G^*}(s)$, except for the automorphisms interchanging the two factors in
type $\tF_4$. But note that here the two factors are generated by long
respectively short root elements, so cannot be interchanged by any automorphism
of $G$ (as $q$ is odd). By \cite[Lemma~3.13]{MNST} there is
$\chi\in\cE(G,st)\cap\Irr(B)$, and the degree formula shows $\chi$ is of
height~1.

All blocks considered here are stable under field automorphisms of $G$, since
the classes of the parametrising isolated elements are. Furthermore, the chosen
characters $\chi\in\Irr(B)$ of height~1 are $\si$-invariant since, as
pointed out above, the class of $t$ is fixed by $\si$. Thus any character
of~$A$ extending $\chi$ will be as desired. The Ennola dual blocks can be
treated by exactly the same argument.

If $S\ne G$ then $G$ necessarily is of type $E_6$ with $\ell=3$ or $E_7$ with
$\ell=2$, $S=G/\bZ(G)$ and $\Irr(B_S)\supseteq\cE(G,s)\cap\Irr(B)$ by deflation.
So for all cases above where we chose $\chi\in\cE(G,s)\cap\Irr(B)$, the centre
of $G$ is in the kernel of $\chi$, so considering $\chi$ as a character of $S$
we conclude as before. In the remaining cases, we have $\ell=2$ and $G$ is of
type $E_6$. Note that here $|\bZ(G)|=3$ is prime to $\ell$, a case we
already discussed.
\end{proof}

We now obtain (a stronger version of) Theorem \ref{thm:qiIntro} on
quasi-isolated blocks of quasi-simple groups:

\begin{thm}   \label{thm:qi diag+fld}
 Let $\ell\ge3$ and $A$ such that $S:=\bF^*(A)$ is quasi-simple of Lie type in
 characteristic $p\ne\ell$. Further, assume that $A$ only involves field
 automorphisms if $S=\tD_4(q)$ and $\ell=3$. Then any $\ell$-block $B$ of $A$
 of non-abelian defect covering a quasi-isolated block of $S$ satisfies
 $\mh(B)=1$.
\end{thm}

\begin{proof}
By Lemma~\ref{lem:Op'} we may assume $A=\bO^{\ell'}(A)$. So $A$ is generated
over $S:=\bF^*(A)$ by outer automorphisms of $\ell$-power order. Our conditions
on $\ell$ mean that these are field automorphisms except possibly if
$S=\tA_n(\eps q)$  or $S=\tE_6(\eps q)$ or $\tD_4(q)$ with $\ell=3$. (Note that
in the remaining cases, field automorphisms of $\ell$-power order necessarily
commute with the diagonal automorphisms.) Let $B$ be an $\ell$-block of $A$ of
non-abelian defect covering a quasi-isolated block of~$S$.

As usual, let $G=\bG^F$ be quasi-simple with $S=G/Z$ for some $Z\le\bZ(G)$.
(The group $S=\tw2\tF_4(2)'$ was considered in Proposition~\ref{prop:spor}.)
Let $B_S$ be a block of $S$ covered by $B$ and $B_G$ a block of $G$
dominating~$B_S$. Note that $B_S$ and $B_G$ determine one another uniquely.
Further, applying \cite[Thms~9.9(c) and~9.14]{N98}, it suffices
to assume that~$Z$ is an $\ell$-group and that $B_S$ is $A$-invariant.
If $A$ only involves field automorphisms, then $\mh(B)=1$ by
Propositions~\ref{prop:unip class}, \ref{prop:unip exc} and
Theorem~\ref{thm:isolated}.

So assume $A$ induces not only field automorphisms. If $S$ is a central
quotient of $\tA_n(\eps q)$, our claim is in
Proposition~\ref{prop:typeA}. Now assume $S$ is a central quotient of
$G=\tE_6(\eps q)$ with $\ell=3\mid(q-\eps)$. Let $\bG\to\tilde\bG$ be a regular
embedding and $\si$ a field automorphism of~$\tiG$ (and hence of $G$ and
$S$) of 3-power order such that $A$ is a quotient of a subgroup of
$\tiG\langle\si\rangle$. By the relevant tables
in \cite{En00,KM13} the only quasi-isolated $3$-blocks $B_G$ of $G$ to consider
are the principal block, the second unipotent block, labelled by the
$e$-cuspidal character of an $e$-split Levi of type $\tD_4$ where $e=e_3(q)$,
and the isolated block labelled by an involution $s\in G^*$ with centraliser of
type $\tA_5\tA_1$. The principal block has been handled in \cite[Thm~3.5]{MMR}.
Now assume $B_G$ is the isolated block. Since there is a unique class of
involutions in $G^*$ with centraliser $\tA_5\tA_1$, the block $B_G$ is
$\si$-invariant, and then so is $\cE(G,s)$. Let $\tilde s\in\tilde\bG^{*F}$ the
involution with image~$s$. Then by the same argument $\cE(\tiG,\tilde s)$
is $\si$-stable. Now $\cE(\tiG,\tilde s)$ contains a character $\chi$ of
height~1 (labelled by a character of height~1 in the relative Weyl group, see
\cite[Tab.~3]{KM13}), uniquely determined
by its degree, with $\bZ(\tiG)$ in its kernel, and with irreducible
restriction to $G$ (since $C_{\bG^*}(s)$ is connected). Thus $\chi$, $\chi|_G$,
and their deflations to $\tiG/\bZ(\tiG)$, $S$ respectively, are
$\si$-stable and so extend to yield a character of $\tiG\langle\si\rangle$
and thus also of $A$ of height~1.

Finally, assume $B_A$ is the non-principal unipotent 3-block of $G$ with
$q\equiv\eps\pmod3$. So, $B_G$ is labelled by the $e$-cuspidal pair $(\bL,\la)$
where $\bL$ has type $\tD_4$ and $e=(3-\eps)/2$. Let $t\in G^*$ be
quasi-isolated of order~3 with centraliser $C_{G^*}(t)=\tD_4(q)(q-1)^2.3$. Then
$t\in [G^*,G^*]$ and hence the characters in $\cE(G,t)$ have $Z\le\bZ(G)$ in
their kernel. Moreover, $t$ has a preimage $\tilde t\in\tiG^*$ of order~3.
Let $\chi$ be an irreducible constituent of $R_L^G(\hat t\la)$, with~$\hat t$
in duality with $t$. Then $\chi\in\cE(G,t)\cap\Irr(B_G)$ is of height~0 (see
also the discussion in the proof of \cite[Lemma~2.14]{KM17}). Since there is a
unique class of quasi-isolated elements in $G^*$ with the given centraliser,
the $\tiG$-orbit of $\chi$ is $\si$-invariant. As $C_{\bG^*}(t)$ is not
connected, $\tchi\in\cE(\tiG,\tilde t)$ has height~1, has $\bZ(G)$ in its
kernel and is $\si$-invariant as again there is a unique class of elements of
order~3 in $\tiG^*$ with centraliser $\tD_4(q).(q-1)^2$. So its deflation to
$\tiG/Z$ extends to $\tiG/Z.\langle\si\rangle$ and provides a character in
$B_A$ as required.
\end{proof}

\begin{rem}
In the case of $G=\tD_4(q)$ with $\ell=3$ and $A$ involving a graph or
graph-field automorphism of 3-power order, the only unipotent 3-blocks of $G$
are the principal block, dealt with in \cite[Thm~3.5]{MMR} and a block of
defect~0. By \cite[Tab.~2]{Ma17} there are five further types of quasi-isolated
3-blocks of $G$, whose irreducible characters coincide with $\cE_3(G,s)$ for
suitable 2-elements $s\in G^*$. For three of these, $s$ are involutions
whose classes are cyclically permuted by the graph automorphism of order~3 but
fixed by field automorphisms, so the corresponding blocks are not $A$-stable,
so also satisfy the conclusion of Theorem~\ref{thm:qi diag+fld}.
The two remaining blocks have $C_{\bG^*}(s)$ of type $\tA_1^4.\tA_1(q)^4.2^2$
and $C_{\bG^*}(s)=A_1(q)T.2^2$ (see e.g.~\cite[Tab.~2]{Ma17}). In both cases,
there are four rational such classes, three of which are permuted transitively
by the graph automorphism, so the corresponding blocks are not $A$-stable.
For the fourth class we have $C_{G^*}(s)=A_1(q)^4.2^2$, respectively
$C_{G^*}(s)=A_1(q)(q-\eps)^2(q+\eps).2^2$ with $q\equiv\eps\pmod4$, a case
we have not been able to resolve.
\end{rem}

\begin{thm}   \label{thm:diag+fld}
 Let $\ell\ge3$ and $A$ such that $\bF^*(A)$ is quasi-simple of Lie type in
 characteristic $p\ne\ell$. Further, assume that $A$ only involves field
 automorphisms if $S=\tD_4(q)$ or $\tE_6(\eps q)$ and $\ell=3$. Then the
 $\ell$-blocks of $A$ are not a minimal counter-example to the direction
 $\mh(B)\le\mh(D)$ of Conjecture~\ref{conj}.
\end{thm}

\begin{proof}
As before, by Lemma~\ref{lem:Op'} we may assume $A=\bO^{\ell'}(A)$, so $A$ is
generated over $S:=\bF^*(A)$ by outer automorphisms of $\ell$-power order.
Our conditions on $\ell$ mean that these are diagonal-field automorphisms.
Moreover, our assumptions ensure that diagonal automorphisms are only involved
when $S$ is a central quotient of $\tA_n(\eps q)$, in which case the
$\ell$-blocks with non-abelian defect are not minimal counter-examples by
Proposition~\ref{prop:typeA}.

Let $B$ be an $\ell$-block of $A$. If $B$ has abelian defect, the claim
certainly holds by Brauer's height zero conjecture. So assume $B$ has
non-abelian defect.

As $S=\tw2\tF_4(2)'$ was considered in Proposition~\ref{prop:spor}, we let
$G=\bG^F$ be quasi-simple with $S=G/Z$ for $Z\le\bZ(G)$,  $B_S$ a block of $S$
covered by $B$ and $B_G$ a block of $G$ dominating~$B_S$. Again, $B_S$ and
$B_G$ determine one another uniquely. Applying \cite[Thm~9.9(c) and~9.14]{N98},
it suffices to assume that~$Z$ is an $\ell$-group and that $B_S$ is
$A$-invariant. By Theorem~\ref{thm:qi diag+fld}, $B_S$ is not quasi-isolated. 

Recall that we have reduced to the case $A$ only involves field automorphisms.
Let $s\in G^*$
be an $\ell'$-element such that $\Irr(B)\subseteq\cE_\ell(G, s)$ and let
$\bL^*\lneq\bG^*$ be the (unique) Levi subgroup minimal with respect to the
property that $C_{\bG^*}(s)\leq \bL^*$, so that $s$ is quasi-isolated in
$\bL^\ast$. Since $B_S$ (and hence $B_G$) is $A/S$-stable, the class of $s$ is
also $A/S$-stable. Then after multiplying by a suitable inner automorphism, we
may choose $\si$ generating $A/S$ such that $C_{G^*}(s)$ is $\si^*$-stable,
where $\si^*$ is the field automorphism of $G^*$ dual to~$\si$. Then our
construction forces $\bL^\ast$ to be $\si^\ast$-stable as well, and hence $\bL$
and $L=\bL^F$ can be taken $\si$-stable. By the Bonnaf\'e--Rouquier reduction
theorem \cite[Thm~B]{BR}, $B_G$ is then Morita equivalent to a block $B_L$
of the strictly smaller group $L$, and $B_L$ and $B_G$ share a defect
group~$D_G$ by \cite[Thm~7.7]{BDR}. More specifically, this Morita equivalence
is given by Lusztig's twisted induction $R_L^G$, which preserves heights, is
equivariant under field automorphisms by \cite[Cor.~3.3.14]{GM20}, and
preserves Lusztig series.

Let $b_L$ be the block of $L/Z$ dominated by $B_L$. Note then that $B_S$ and
$b_L$ share the defect group $D_GZ/Z$.
Since $Z$ is central, the characters of $B_S$ and $b_L$ are exactly those
in~$B_G$, respectively $B_L$, that are trivial on $Z$. Since $R_L^G$ preserves
Lusztig series and all characters in the same Lusztig series lie above the same
character of $\bZ(G)$ (see \cite[Lemma~2.2]{Ma07}), we see that our bijection
also induces a height-preserving bijection from $\Irr(B_S)$ to $\Irr(b_L)$.
Hence as $A/S$ is cyclic, involving only field automorphisms, it follows 
that $B$ cannot be a minimal counter-example, noting that $B$ and the block
above $b_L$ in $L\langle\si\rangle$ share a defect group.
\end{proof}

With this and the results in the previous sections, we obtain
Theorem~\ref{thm:aut}. Further, our methods have also proved
Theorem~\ref{thm:quasi}:

\begin{cor}   \label{cor:nondefnoauts}
 Theorems~\ref{thm:quasi} and~\ref{thm:aut} from the Introduction hold.
\end{cor}

\begin{proof}
Thanks to Theorem~\ref{thm:diag+fld}, together with Theorem~\ref{thm:defchar}, 
and the results and references in Section \ref{sec:nonlie}, we see that it now
suffices to show that when $\ell=2$, any $2$-block $B$ of a quasi-simple group
$S$ of Lie type defined in odd characteristic $p$ is not a minimal
counter-example to $\mh(B)\leq \mh(D)$. 

Now, we may assume by the established Brauer's height zero conjecture for
quasi-simple groups \cite{KM17} that $D$ is non-abelian, and by
Theorem~\ref{thm:isolated} that the block $B$ is not quasi-isolated. From here,
the same argument as in the proof of Theorem~\ref{thm:diag+fld}, using
Bonnaf{\'e}--Rouquier's reduction theorem, yields the result.
\end{proof}

\begin{rem}
We close by noting that for a block $B$ of a finite group $G$ with defect
group $D$, it is not in general the case that the number of characters of the
smallest non-zero height for $B$ and for $N_G(D)$ would agree (while for
characters of height zero this is postulated by the Alperin--McKay conjecture).
An example is given by the principal 2-block $B$ of $G=M_{12}$, where
$N_G(D)=D$ has six characters of height~1 while $B$ only has two.
\end{rem}


\end{document}